\documentclass[10pt]{article}
\usepackage[margin=1in]{geometry}

\usepackage{psfrag}
\usepackage[numbers]{natbib}        
\usepackage{hyperref}
\usepackage[normalem]{ulem}
\usepackage{overpic}

\usepackage{enumerate}
\usepackage{euscript}
\usepackage{graphicx}
\usepackage{amssymb}
\usepackage{amsmath}
\usepackage{amsthm}
\usepackage{dsfont}
\usepackage{color}
\usepackage{breqn}
\newtheorem{assumption}{Assumption}
\newtheorem{definition}{Definition}
\newtheorem{lem}{Lemma}
\newtheorem{rem}{Remark}

\newtheorem{theorem}{Theorem}

\newcommand{\E}{\mathrm{E}}

\newcommand{\R}{\mathrm{Re}}

\def\mk#1{{\color{black}#1}}

\def\hb{\boldsymbol{h}}

\def\R{\mathbb{R}}

\def\Z{\mathbb{Z}}

\def\xb{\mathbf{x}}

\def\cb{\mathbf{c}}

\def\zbx{\mathbf{X}}
\def\ab{\mathbf{a}}

\def\yb{\mathbf{y}}
\def\bmu{\boldsymbol{\mu}}
\def\be{\boldsymbol{\eta}}
\def\Mb{\boldsymbol{M}}
\def\Qb{\boldsymbol{Q}}
\def\bla{\boldsymbol{\lambda}}
\def\bsigma{\boldsymbol{\sigma}}
\def\Gb{\boldsymbol{G}}
\def\EQ{\EuScript Q}
\def\Rb{\boldsymbol{R}}
\def\Ab{\boldsymbol{A}}
\def\bemax{\beta_{\mbox{\scriptsize max}}}
\def\bemin{\beta_{\mbox{\scriptsize min}}}
\def\sigmax{\sigma_{\mbox{\scriptsize max}}}
\def\gamax{\gamma_{\mbox{\scriptsize max}}}
\def\cb{\mathbf{c}}

\def\gb{\mathbf{g}}
\def\R{\mathbb{R}}

\def\Z{\mathbb{Z}}

\def\xb{\mathbf{x}}

\def\zbx{\mathbf{X}}
\def\ab{\boldsymbol{a}}
\def\mb{\boldsymbol{m}}
\def\by{\boldsymbol{y}}
\def\cb{\mathbf{c}}

\def\bx{\boldsymbol{x}}
\def\Proj{\mbox{Proj}}

\def\Lag{\EuScript L}
\def\r{~}

\begin{document}
%
%

\title{\LARGE
Learning Generalized Nash Equilibria in a Class of Convex Games\let\thefootnote\relax\footnotetext{This research was gratefully funded by the European Union ERC Starting Grant CONENE.}\\
}
\author{Tatiana Tatarenko\thanks{Department of Control Theory and Robotics, TU Darmstadt, Germany} \and Maryam Kamgarpour\thanks{Automatic Control Laboratory, ETH Z\"urich, Switzerland}\footnotemark[1]
}

\maketitle

\begin{abstract}
We consider multi-agent decision making where each agent optimizes its convex cost function subject to individual and coupling constraints. The constraint sets are compact convex subsets of a Euclidean space. To learn Nash equilibria, we propose a novel distributed payoff-based algorithm, where each agent uses  information only about its cost value and the constraint value with its associated dual multiplier. We prove convergence of this algorithm to a Nash equilibrium, \mk{under the assumption that the game admits a strictly convex potential function}. \mk{In the absence of coupling constraints, we prove convergence to Nash equilibria under significantly weaker assumptions, not requiring a potential function. Namely, strict monotonicity of the game mapping is sufficient for convergence.} We also derive the convergence rate of the algorithm for strongly monotone game maps.

\end{abstract}
%

\section{Introduction}

Decision making in multi-agent systems arises in engineering applications ranging from electricity markets to telecommunication and transportation networks \cite{Vehicle, BasharSG, Scutaricdma}. Game theory provides a powerful framework for analyzing and optimizing decisions in multi-agent systems. The  notion of an equilibrium in a game  characterizes stable solutions to multi-agent decision making problems. In this work, we design a distributed learning algorithm to converge to Nash equilibria for a class of non-cooperative games modeled by convex objective functions and coupling constraints.

There is a large body of work on computation of Nash equilibria. The approaches differ mainly by the particular structure of agents' cost functions as well as the information available to each agent. In a so-called \textit{potential game},  a central optimization problem can be formulated whose minimizers coincide with a subset of the Nash equilibria of the game. One can then leverage  distributed optimization algorithms to compute the minima of the potential function \cite{MardDes, gossipLacra}, despite agents' limited information of others' cost functions or action sets. Distributed algorithms have also been designed for the class of  \textit{aggregative games} \cite{Jensen, paccagnan2016aggregative}. In general, for implementation of the aforementioned distributed algorithms each agent needs to know the structure of its cost function or its derivative. Furthermore, agents may need to communicate with each other or with a central entity, even if their strategy spaces are  
decoupled.

In contrast to deterministic distributed optimization approaches, learning approaches start with the assumption that the each agent's objective function or its derivative may not be known  to the agent itself nor to the other agents. They attempt to compute Nash equilibria by sampling agents' actions from a set of probability distributions. These probability distributions are updated based on the information available in the system. In particular, most of the past work has focused on algorithms that require the ability of each agent to evaluate  its cost function at any feasible point, given fixed actions of all other agents, and convergence is established for the subclass of potential games  \cite{MardenLLL, Leslie, Tat_ACC14, Tat_ECC16}.

In many practical situations the agents do not know functional form of their objectives and can only access the values of their objective functions at a played action. Such situations arise, for example,  in electricity markets (unknown price functions or constraints) \cite{windfarm,elmark}, network routing (unknown traffic demands/constraints) \cite{netrout, netrout1}, and sensor coverage problems (unknown density function on the mission space) \cite{COVEr}. In such cases, the information structure is referred to as \textit{payoff-based}, that is, each agent can only observe its obtained payoffs and be aware of its local actions. A payoff-based learning algorithm in potential games is proposed in \cite{MardRev} with the guarantee of stochastic stability of potential function minimizers. However, to implement this payoff-based algorithm agents need to have some memory. Other algorithms requiring only payoff-based information and memory are proposed in \cite{Goto_PIPIP} and \cite{COVEr}. These learning procedures 
assume a potential game and guarantee convergence to a distribution over potential function
minimizers in total variation. In \cite{Shamma2015} the idea of dynamic feedback is utilized for matrix games and an extension of fictitious play is proposed that considers empirical frequencies and their derivatives. The convergence to Nash equilibrium in this setting is established.  Learning based approaches are also proposed in \cite{pradelski2012learning} for non-potential games, where  stochastic convergence to the Nash equilibrium maximizing social welfare is guaranteed.

The above payoff-based procedures are applicable to games with finite action spaces. For games with uncountable action spaces, a payoff-based approach was developed based on extremum seeking \cite{frihauf2012nash}. The extremum seeking approach designs a dynamic update law for the actions based on sinusoidally perturbed measured payoffs. If the amplitude and frequency of this sinusoid are chosen properly, locally asymptotically stable equilibria of this dynamical system will correspond to Nash equilibria of the game.  This approach was extended to account for stochastic noise affecting measurements of the cost functions  \cite{stankovic2012distributed}. Given strongly convex cost functions almost sure convergence to a Nash equilibrium was proven. An alternative payoff-based approach, inspired by the \textit{logit} dynamics in finite action games \cite{blume1993statistical}, was proposed in \cite{Tat_cdc16} in a potential game setting. This approach considered sampling agents' actions from a Gaussian 
distribution. The result was generalized to arbitrary games (not necessarily potential) with uncoupled action sets in \cite{Tat_ifac2017}.

Despite considerable progress in learning algorithms for games, the work on payoff-based learning has not considered convex cost functions and coupling constraints on agents' actions. In several realistic scenarios in which players share resources each player's feasible strategy space depends on the other players' actions. For example, in an electricity market, there are  coupling constraints due to the underlying physical electricity network. Similar constraints exist in a transportation or telecommunication network and general deregulated economy problems \cite{scutari2012monotone}. The Nash equilibria in a game with coupling constraints are referred to as \emph{generalized Nash equilibria}. Ensuring uniqueness of these equilibria and computing them is a lively research topic \cite{facchinei2007generalized}.

We focus on learning equilibria in a subset of generalized Nash equilibrium problems in which the coupling constraint is shared among the agents and is jointly convex (convex in all actions). In this setting, one can formulate a variational equilibrium problem  to characterize a subset of the generalized Nash equilibria, referred to as variational equilibria \cite{facchinei2007generalized}.  In addition to computational advantages, variational equilibria present  few desired practical properties. For example, since in this equilibrium, the dual multiplier associated to the joint constraint is equal across all players, there is a well-defined cost associated to constraint violation. Note that variational equilibria are also a subset of the normalized equilibria, introduced in the seminal work \cite{Rosen1965}, where the normalizing coefficients considered in \cite{Rosen1965} would be constant across all players.
The authors in \cite{kulkarni2009new} further derived theoretical  connections between  variational equilibria and generalized Nash equilibria and showed that in the interior of the shared constraint sets, the two equilibria concepts are equivalent.

Recent research has focused on distributed algorithms for computing generalized Nash equilibria. Authors in \cite{Shanbhag2011} consider variational equilibria in monotone games and propose a primal-dual distributed algorithm. Similarly, \cite{dario2016cdc,gentile2017nash} addresses decentralized computation of variational equilibria for aggregative games. In \cite{salehisadaghiani2016distributed} a distributed primal-dual algorithm for computing generalized Nash equilibria is proposed for a network game. In the network game setting, it is assumed that there exists a communication graph through which each player can share its strategy information with its neighbors. Hence, some coordination between agents is needed. In all the above  work, each agent needs to know the functional form of their cost function or its gradient. The work in  \cite{ZhuFrazzoli} suppresses this requirement and develops a primal-dual algorithm for learning generalized Nash equilibria. Nevertheless, players need to 
exchange information with their neighbors according to the network graph. As such,  they can  estimate the gradients of their  cost functions online using  neighborhood information.

The approach to estimate the gradient of a cost function online is  well-studied in the stochastic optimization literature \cite{nesterov2011random}. In the game setting however, this approach necessitates some coordination between agents. This is because for a given player to estimate the gradient of its  cost function, it needs to evaluate this cost function at at least two points of its strategy space while other agents who influence the player's cost function should not change their actions (otherwise,  the player cannot attribute the decrease/increase of its cost to its own actions and hence, the gradient cannot be estimated). Our goal is to develop a payoff-based algorithm that bypasses the need for coordination or information exchange during each step of the algorithm. Naturally however, similar to all past algorithms, the agents must agree to implement the algorithm.

Our contributions are as follows. First, we develop a payoff-based approach for computing  Nash equilibria in a class of convex games with jointly convex coupling constraints. Second, we prove almost sure convergence of the algorithm to variational Nash equilibria, \mk{under the existence of a strictly convex potential function}. \mk{Third, in the absence of coupling constraints, we prove  almost sure convergence to variational Nash equilibria, relaxing the requirement of existence of a potential function}. Fourth, for this latter setup, we quantify the convergence rate of the payoff-based algorithm if the game map is strongly monotone. While our setup is similar to \cite{Rosen1965, Shanbhag2011, ZhuFrazzoli,dario2016cdc}, in contrast to the above work we do not require knowledge of the cost functions, constraints or their gradients \cite{dario2016cdc,Shanbhag2011}. Also,  we require neither information exchange between players \cite{salehisadaghiani2016distributed,ZhuFrazzoli}, nor knowledge of a norm bound on the dual multipliers of the coupling constraints \cite{ZhuFrazzoli}. 

Our approach is detailed as follows. We extend the  game to define a player corresponding to the dual multiplier of the coupling constraints, similar to \cite{Shanbhag2011, ZhuFrazzoli,dario2016cdc,gentile2017nash}. We then develop a novel  sampling based approach, in which the probability distributions from which agents sample their actions are Gaussian, inspired by the literature on learning automata \cite{Thatha}. The mean of the distribution is updated iteratively by each agent based only on its own current payoff and local constraint set.  The dual player, on the other hand, updates its action deterministically by measuring constraint violation at each time step. Notice that similar to \cite{dario2016cdc, ZhuFrazzoli} the dual player is a fictitious player. It can refer to a central coordinator who measures the constraint violation at each step. Alternatively, if each agent can locally measure the
constraint violation, it can update its dual variable. Furthermore, similar to primal-dual algorithms in \cite{dario2016cdc, Shanbhag2011, ZhuFrazzoli} constraints are satisfied upon convergence of the algorithm.  To prove convergence of our algorithm we leverage results on Robbins-Monro stochastic approximation \cite{Borkar,NH}. We quantify the convergence rate based on rate estimates in stochastic projection algorithms \cite{stochprogr}.

This paper is organized as follows. In Section \ref{sec:problem},  we set up the game under consideration. In Section \ref{sec:analysis}, we propose our payoff-based approach and present its convergence result. Section \ref{sec:proof} develops the proof of the main result using supporting theorems on stochastic random variables. In Section \ref{sec:rate}, we \mk{relax the coupling constraint and consequently, the requirements for convergence of the proposed algorithm.} Furthermore, we provide a convergence rate for this latter case. A  case study is provided in Section \ref{sec:simulations} based on a game arising in a classical Cournot economic model. In Section \ref{sec:conclusion}, we summarize the result and discuss future work.

\textbf{Notations and basic definitions.} The set $\{1,\ldots,N\}$ is denoted by $[N]$. Boldface is used to distinguish between vectors in a multi-dimensional space and scalars.
Given $N$ vectors $\bx^i\in\R^d$, $i\in[N]$, $[\bx^i]_{i=1}^{N}:=[{\bx^1}^{\top}, \ldots, {\bx^N}^{\top}]^{\top} \in \R^{Nd}$; $\bx^{-i}:=[{\bx^1}, \ldots, {\bx^{i-1}},{\bx^{i+1}}, \ldots, {\bx^N}] \in \R^{(N-1)d}$.
$\R^d_{+}$ and $\Z_{+}$ denote respectively, vectors from $\R^d$ with non-negative coordinates and non-negative whole numbers.  The standard inner product on $\R^d$ is denoted by $(\cdot,\cdot)$: $\R^d \times \R^d \to \R$, with associated norm $\|\bx\|:=\sqrt{(\bx, \bx)}$. We let $\R^d_{\le K}=\{\xb\in\R^d: \|\xb\|\le K\}$. $I_d$ represents the $d$-dimensional identity matrix and $\mathds{1}_N$ represents the $N$-dimensional vector of unit entries. Given some matrix $A\in\R^{d\times d}$, $A\succeq(\succ)0$, if and only if $\bx^{\top}A\bx\ge(>)0$ for all $\bx\ne 0$.

 Given a function $\gb(\bx, \by):\R^{d_1}\times\R^{d_2}\to\R$, we define the
 mapping $\nabla_{\bx}\gb(\bx, \by): \R^{d_1}\times\R^{d_2}\to\R^{d_1}$ component-wise as $[\nabla_{\bx}\gb(\bx, \by)]_{i}:=\frac{\partial \gb(\bx, \by)}{\partial x^i}$.
 We use the big-$O$ notation, that is, the function $f(x): \R\to\R$ is $O(\gb(x))$ as $x\to a$, $f(x)$ = $O(g(x))$ as $x\to a$, if $\lim_{x\to a}\frac{|f(x)|}{|g(x)|}\le K$ for some positive constant $K$.
 We say that a function $f(\bx)$ grows not faster than a function $g(\bx)$ as $\bx\to\infty$, if there exists a positive constant $Q$ such that $f(\bx)\le g(\bx)$ for any $\bx$ with $\|\bx\|\ge Q$.

\begin{definition}\label{def:pmm}
 The mapping $\Mb:\R^d\to\R^d$ is called \emph{pseudo-monotone} over $X\subseteq\R^d$, if $(\Mb(\by),\bx-\by)\ge 0$ implies $(\Mb(\bx),\bx-\by)\ge 0$ for every $\bx,\by\in X$.
 \end{definition}

\begin{definition}\label{def:strongmon}
 The mapping $\Mb:\R^d\to\R^d$ is called \emph{strongly monotone} over $X\subseteq\R^d$ with constant $\kappa>0$, if $(\Mb(\bx)-\Mb(\by),\bx-\by)\ge \kappa\|\bx-\by\|^2$ for any $\bx,\by\in X$. It is \emph{strictly monotone} if $(\Mb(\bx)-\Mb(\by),\bx-\by)> 0$ for any $\bx,\by\in X$.
 \end{definition}

%
\begin{definition}\label{def:slater}
Let $\EuScript C$ be a convex constraint set described by  a finite set of convex inequality constraints $\EuScript C=\{\bx\in\R^d: f^i(\bx)\le 0, i\in[m]\}$. The \emph{Slater's constraint qualification} consists in existence of a  strictly feasible point $\xb^* \in \EuScript C$, $f^i(\bx^*) < 0$ for $ i \in [m]$.
 \end{definition}

\section{Problem Formulation}\label{sec:problem}
\subsection{Convex games with coupling constraints}
We consider a game $\Gamma (N, \{A_i\}, \{J_i\}, C)$ with $N$ players. We assume that the action of the $i$th player is locally constrained to $\ab^i\in A_i\subset \R^d$ and that the vector of joint actions\footnote{All results below are applicable for games with different dimensions $\{d_i\}$ of the action sets $\{A_i\}$.}, $\ab = [\ab^1 \ldots, \ab^N ]\in \Ab = A_1\times\ldots\times A_N$, has to belong to a \emph{global coupling constraint set} $C$, namely
\begin{align}\label{eq:globconstr}
 \ab\in C=\{\ab\in\Ab: \gb(\ab)\le \boldsymbol 0\},
\end{align}
where $\gb:\R^{Nd}\to\R^n$ with coordinates $g_i(\ab), i\in[n]$.
Let $\EQ = \Ab\cap C$, $\EQ^i(\ab^{-i})=\{\ab^i\in A^i: \gb(\ab^i,\ab^{-i})\le 0\}$. The cost functions $J_i:\R^{Nd}\to\R$ indicate the cost $J_i(\ab)$ the agent $i$ has to pay, given any joint action $\ab\in\EQ$. Throughout this paper we assume $A_i$ to be compact for all $i\in[N]$.

A \emph{generalized Nash equilibrium} (GNE) in a game $\Gamma$ with coupled actions represents a joint action from which no player has any incentive to unilaterally deviate.
\begin{definition}\label{def:GNE}
 A point $\ab^*\in\EQ$ is called a \emph{generalized Nash equilibrium} (GNE) if for any $i\in[N]$ and $\ab^i\in \EQ^i(\ab^{*-i})$
 $$J_i(\ab^{*i},\ab^{*-i})\le J_i(\ab^{i},\ab^{*-i}).$$
 If $C=\R^{Nd}$ then $\EQ^i(\ab^{-i})=\{a^i: a^i\in A_i\}$ and any $\ab^*$ for which the inequality above holds is a \emph{Nash equilibrium (NE)}.
 \end{definition}

We consider convex games as follows. 
\begin{assumption}\label{assum:convex}
 The game under consideration is \emph{convex}. Namely, for all $i\in[N]$ the set $A_i$ is convex and compact, the cost function $J_i(\ab^i, \ab^{-i})$ is defined on $\R^{Nd}$, continuously differentiable in $\ab$ and convex in $\ab^i$ for  fixed $\ab^{-i}$. The coupling constraint function $\gb:\R^{Nd}\to\R^n$ is continuously differentiable and has convex coordinates $g_i(\ab), i\in[n]$.
\end{assumption}

Given differentiable cost functions, we define the game mapping and the extended game mapping.
\begin{definition}\label{def:mapping}
The mapping $\Mb:\R^{Nd}\to\R^{Nd}$, referred to as  the \emph{game mapping} of  $\Gamma(N, \{A_i\}, \{J_i\}, C)$ is defined by
 \begin{align}\label{eq:gamemapping}
 &\Mb(\ab) = \cr
 &[M_{1,1}(\ab), \ldots, M_{1,d}(\ab),\ldots, M_{N,1}(\ab), \ldots, M_{N,d}(\ab)]^{\top}, \cr
 &M_{i,k}(\ab)= \frac{\partial J_i(\ab)}{\partial a^i_k}, \; \ab\in\EQ=\Ab\cap C, i\in[N], k\in[d].
 \end{align}
\end{definition}

\begin{definition}\label{def:mapping_ext}
The mapping $\Mb^0:\R^{Nd+n}\to\R^{Nd+n}$, referred to as  the \emph{extended game mapping} of  $\Gamma(N, \{A_i\}, \{J_i\}, C)$ with coupled actions, is defined by
 \begin{align}\label{eq:gamemappingext}
 &\Mb^0(\ab, \bla)= [\Mb^0_{1}(\ab, \bla),\ldots, \Mb^0_{N}(\ab, \bla),-\gb(\ab)]^{\top}, \cr
 &\Mb^0_{i}(\ab, \bla)=[M^0_{i,1}(\ab, \bla), \ldots, M^0_{i,d}(\ab, \bla)], \quad i\in[N],\cr
 &M^0_{i,k}(\ab,\bla)= M_{i,k}(\ab) +\frac{\partial \big(\bla, \gb(\ab)\big)}{\partial a^i_k},\; \ab\in\EQ=\Ab\cap C, \cr
 &\qquad\qquad\qquad\qquad\qquad\qquad\qquad i\in[N], \quad k\in[d].
 \end{align}
\end{definition}


To design an algorithm with bounded iterates, we  need the following standard assumptions \cite{dario2016cdc, Shanbhag2011, ZhuFrazzoli, facchinei2007generalized}.
\begin{assumption}\label{assum:CG_grad}
 The coordinates $\Mb_i^0(\ab, \bla):\R^{Nd+n}\to\R^{d}$ of extended mapping $\Mb^0(\ab, \bla):\R^{Nd+n}\to\R^{Nd+n}$ of a game $\Gamma(N, \{A_i\}, \{J_i\}, C)$ with coupled actions
 are \emph{Lipschitz on $\R^{Nd}$} with respect to coordinates $\ab$ with a linear function $L_i(\bla)$ . The function $\gb(\ab)$ is Lipschitz on $\R^{Nd}$. Moreover, the extended game mapping $\Mb^0(\ab, \bla)$ is \emph{pseudo-monotone on $\EQ\times\R^n_+=(\Ab\cap C)\times\R^n_+$}.
 \end{assumption}
 
  \begin{assumption}\label{assum:Slaters}
The sets $A_i$, $i\in[N]$, $\Ab$, and $\EQ$ satisfy the Slater's constraint qualification (see Definition\r\ref{def:slater}).
\end{assumption}

 \begin{assumption}\label{assum:inftybeh}
 The cost functions $J_i(\ab)$, $i\in[N]$, grow not faster than a quadratic function of $\ab^i$ as $\|\ab^i\|\to\infty$.
\end{assumption}

Let us provide some insight on the assumptions above. 
   \begin{rem}\label{rem:pseudomon}
If agents' cost functions are quadratic and coupling constraints are linear, the extended game mapping is affine, namely $\Mb^0(\ab, \bla) = M[\ab, \bla]^{\top} + \mb$, where $M \in \R^{(Nd+n)\times (Nd+n)}$ and $\mb \in \R^{Nd+n}$. The affine mapping above is pseudo-monotone if $M$ is positive semi-definite \cite{AffinePsedoMonot}. This is in particular fulfilled if the quadratic forms of the cost functions are positive definite or semi-definite (see \cite{paccagnan2016aggregative} and \cite{Tat_ifac2017}, respectively). However, if the affine map is pseudo-monotone for every $\mb$ then $M$ is positive semi-definite and hence the map is also monotone \cite{AffinePsedoMonot}. In  general, monotonicity implies pseudo-monotonicity and the  former is more stringent\footnote{As an example, gradient of any pseudo-convex function such as $x^3 + x$ is pseudo-monotone but not necessarily monotone.}.
 \end{rem}
 
  \begin{rem}\label{rem:lipsch}
Since the extended mapping $\Mb^0(\ab,\bla)$ is affine in $\bla$, the Lipschitz condition for $\Mb^0(\ab,\bla)$ in Assumption\r\ref{assum:CG_grad} above holds if the coordinates $\Mb_i(\ab)$, $i\in[N]$, of game mapping $\Mb(\ab)$ and the functions $\frac{\partial g_j(\ab)}{\partial a^i_k}$, $j\in[n]$, $i\in[N]$, $k\in[d]$, are Lipschitz with respect to their  argument $\ab=(\ab^1,\ldots,\ab^N)$ with some constants $l_i$, $l_{j,i}^k$ respectively.
 \end{rem}

 \begin{rem}\label{rem:fun_g}
Given Lipschitz continuity of $\gb$ on $\R^{Nd}$, the functions $g_i(\ab), i\in[n]$, grow not faster than a linear function of $\ab$ as $\|\ab\|\to\infty$. Furthermore, since the action sets are compact, we can always approximate $J_i$ outside $A_i$ by a quadratic function without loss of generality.
 \end{rem}

\subsection{Generalized Nash equilibria and Variational Inequalities}\label{subsec:problem_opt}
Here, we prove that the set of GNE is nonempty, given fulfillment of Assumptions\r\ref{assum:convex}-\ref{assum:Slaters} for the game $\Gamma (N, \{A_i\}, \{J_i\}, C)$. This result will be obtained through connecting generalized Nash equilibria and solutions of variational inequalities. Moreover, we model an uncoupled action game associated with the game $\Gamma$ and establish the relation between its Nash equilibria and the GNE of the game $\Gamma$. Existence of such an uncoupled action game will allow us to present a payoff-based approach to learning GNE in the initial game $\Gamma$.

\begin{definition}
Consider a mapping $\boldsymbol T(\cdot)$: $\R^d \to \R^d$ and a set $Y \subseteq \R^d$. The
\emph{solution set $SOL(Y,\boldsymbol T)$ to the variational inequality problem} $VI(Y,\boldsymbol T)$ is a set of vectors $\yb^* \in Y$ such that $(\boldsymbol T(\yb^*), \yb-\yb^*) \ge 0$, for all $\yb \in Y$.
\end{definition}
Given $VI(Y,\boldsymbol T)$, suppose that the set $Y$ is compact, convex and that the mapping $\boldsymbol T$ is continuous. Then, $SOL(Y,\boldsymbol T)$ is nonempty and compact (see Corollary 2.2.5 in \cite{FaccPang1}. )

For a game $\Gamma (N, \{A_i\}, \{J_i\}, C)$ with coupled actions $\EQ=\Ab\cap C$, we can define $VI(\EQ,\boldsymbol M)$, where $\Mb$ is the game mapping defined in \eqref{eq:gamemapping}. Under  Assumption\r\ref{assum:convex}, Theorem\r2.1 in \cite{FaccFischer} implies that  $SOL(\EQ,\boldsymbol M)$ is non-empty and if $\ab^*\in SOL(\EQ,\boldsymbol M)$, then $\ab^*$ is a GNE in the game $\Gamma$. 

A challenge in developing a payoff-based learning based algorithm lies in the coupling constraint $C$. Hence,  our goal is to develop a game with uncoupled actions whose equilibria can be used to find those of  the original game $\Gamma$. To do so,  we first define an \emph{associated game} $\Gamma_{a}(\Ab\times \R^n_+)$ as follows:
 \begin{align}\label{eq:assocgame1}
  \Gamma_{a}(\Ab\times \R^n_+)=\Gamma_{a}(N+1,\{J^0_i\}_{i\in[N+1]}, \{\{A_i\}_{i\in[N]}, \R^n_+\}),
 \end{align}
 with $N+1$ players. The first $N$ players are called regular and the $(N+1)$th player is called dual. The action sets of the regular players coincide with the local action sets $\{A_i\}$ of the players in the initial game $\Gamma$, whereas the action set of the dual player is the set $\R_+^{n}$.
 The cost functions of the players in $\Gamma_{a}(\Ab\times \R_+^{n})$ are defined as follows:
 \begin{align}
 \nonumber
J^0_i(\ab^i,\ab^{-i},\bla)&=J_i(\ab^i,\ab^{-i})+(\bla, \gb(\ab^{i},\ab^{-i})), \quad i\in[N],\\
 \label{eq:assocgame2}
J^0_{N+1}(\ab,\bla)&=-(\bla, \gb(\ab)).
 \end{align}
So, the cost function of each regular player $i\in[N]$ in the game $\Gamma_{a}(\Ab\times \R_+^{n})$ is composed of two terms: the original cost function from the game $\Gamma$ plus an additional term that depends on the strategy $\bla$ of the dual player and on the influence of the current joint action in the coupling constraint. As $\bla\ge\boldsymbol 0$, the latter can be interpreted as a term penalizing violations of the global constraint by the given joint action.
%
\begin{lem}\label{th:exist_uncoupled}
 Let $\Gamma (N, \{A_i\}, \{J_i\}, C)$ be a game for which Assumptions\r\ref{assum:convex}-\ref{assum:Slaters} hold.  Then,
 \begin{itemize}
 \item[1)]  $[\ab^*, \bla^*]\in \Ab\times \R_+^{n}$ is a Nash equilibrium in $\Gamma_{a}(\Ab\times \R_+^{n})$, if and only if $[\ab^*,\bla^*]\in SOL(\Ab\times \R_+^{n},\boldsymbol M^0)$,
 \item[2)] if  $[\ab^*,\bla^*]$ is a Nash equilibrium in $\Gamma_{a}(\Ab\times \R_+^{n})$, then $\ab^*$ is a GNE of $\Gamma$,
 \item[3)] there exists a Nash equilibrium $[\ab^*, \bla^*]$ in $\Gamma_{a}(\Ab\times \R_+^{n})$,
 \item[4)] for any Nash equilibrium $[\ab^*, \bla^*]$ in $\Gamma_{a}(\Ab\times \R_+^{n})$ there exists a constant $K > 0$ such that $\|\bla^*\|\le K$.

 \end{itemize}
 \end{lem}
Please refer to the appendix for the proof of the above lemma.


\section{Payoff-Based Algorithm}\label{sec:analysis}

Let $\xb^i(t)=[x^i_1(t),\ldots,x^i_d(t)]^{\top}\in\R^{d}$ denote the stra\-tegy of player $i$ at iteration $t$ referred to as its \textit{state} and $\hat J^0_i(t) = J^0_i(\xb^1(t),\ldots,\xb^N(t),\bla(t))$ the current value of its cost at a joint state. Each regular player $i\in[N]$, ``mixes'' its  state, namely, it chooses its  state $\xb^i(t)$ randomly according to the multidimensional normal distribution $\EuScript N(\bmu^i(t)=[\mu^i_1(t),\ldots,\mu^i_{d}(t)]^{\top},\sigma_i(t))$ with density:
\begin{align*}
 p_i&(x^i_1,\ldots,x^i_{d};\bmu^i(t),\sigma_i(t))\cr
 & = \frac{1}{(\sqrt{2\pi}\sigma_i(t))^{d}}\exp\left\{-\sum_{k=1}^{d}\frac{(x^i_k-\mu^i_k(t))^2}{2\sigma_i^2(t)}\right\},
\end{align*}
where $i\in[N]$.
Our choice of Gaussian distribution is based on the idea of  Continuous Action-set Learning Automaton presented in the literature on learning automata \cite{Thatha}.

The mean $\bmu^i(t)$ of the state's distribution is considered an action of the regular agent $i$ and is updated as follows:
\begin{align}
\label{eq:regpl}
 &\bmu^i (t+1)=\\
 \nonumber
 & \Proj_{A_i}\left[\bmu^i(t) -\gamma_i(t+1)\sigma_i^2(t+1){\hat J^0_i(t)}\frac{{\xb^i(t)} -\bmu^i(t)}{\sigma_i^2(t)} \right],
 \end{align}
where $i\in[N]$, $\gamma_i(t+1)$ is a step-size parameter chosen by  player $i$ and $\Proj_{A_i}[\cdot]$ denotes the projection on set $A_i$. The initial value of $\bmu(0)$ can be set to any finite value arbitrarily.

As for the dual player $N+1$, it updates its current action $\bla(t)$ based only on the observation of the violation of the constraint $C$, namely based on the actual value $\hat \gb(t)$, of the function $\gb(\xb^1(t),\ldots,\xb^N(t))$ at the current joint state of the regular players as follows:
\begin{align}\label{eq:dualpl}
 \bla(t+1)=\Proj_{\R^n_{+}}[\bla(t) + \beta_0(t+1)\hat \gb(t)],
\end{align}
where $\beta_0(t+1)$ is a step-size parameter chosen by the dual player $N+1$.
The initial  value of $\bla(0)$ can be arbitrarily set to any finite value. 

Note that in contrast to the approach in computing generalized Nash equilibria presented in \cite{ZhuFrazzoli}, our proposed payoff-based algorithm does not rely on the specified bound $K$ of the dual variable $\bla^*$ in the associated bounded game $\Gamma_{ab}(\Ab\times\R^n_{\le K+r})$, nor does it assume a communication graph between agents to estimate local gradients.

To analyze the convergence of this algorithm,   we show that this algorithm is analogous to the Robbins-Monro stochastic approximation procedure \cite{Borkar}. 

Given $\bsigma=(\sigma_1,\ldots,\sigma_N)$, for any $j\in[N+1]$ define
 \begin{align*}
  \tilde{J}^0_j &(\bmu^1,\ldots,\bmu^N, \bla, \bsigma)= \int_{\mathbb R^{Nd}}J^0_j(\bx,\bla)p(\bmu, \bx, \bsigma)d\bx,  \cr
  &p(\bmu, \bx, \bsigma)=\prod_{j=1}^Np_j(x^j_1,\ldots,x^j_{d};\bmu^j,\sigma_j).
 \end{align*}
For $j\in[N+1]$, $\tilde{J}^0_j$ can be interpreted as the $j$th player's cost function in mixed strategies of the regular players, given that the mixed strategies of these players are multivariate normal distributions $\{\EuScript N(\bmu^i,\sigma_i)\}_{i\in[N]}$. It follows that the second terms inside the projections in \eqref{eq:regpl} and \eqref{eq:dualpl} are samples of the gradient of the cost function with respect to these mixed strategies. In particular, we can verify that for all $\bmu\in\Ab$\footnote{Assumption\r\ref{assum:inftybeh} and compact set $\Ab$ justify differentiation under the integral sign of $\tilde J^0_i(\bmu^1(t),\ldots,\bmu^N(t), \bla(t), \bsigma(t))$.}
\begin{align}
\label{eq:mathexp}
 &\E_{\xb(t)}\{\hat J^0_i(t)\frac{x^i_k(t) -\mu^i_k(t)}{\sigma_i^2(t)}\} \\
 \nonumber
 =&\E\{J^0_i(\xb^1(t),\ldots,\xb^N(t),\bla(t))\frac{x^i_k(t) -\mu^i_k(t)}{\sigma_i^2(t)}|\\
 \nonumber
&\qquad\qquad\qquad\quad x^i_k(t)\sim\EuScript N(\mu_k^i(t),\sigma_i(t)), i\in[N], k\in[d]\}\\
 \nonumber
 = & \frac{\partial {\tilde J^0_i(\bmu^1(t),\ldots,\bmu^N(t), \bla(t), \bsigma(t))}}{\partial \mu^i_k},\; \forall i\in[N], k\in[d],\\
\label{eq:mathexp1}
 & \E_{\xb(t)}\{\hat \gb(t)\}=\E\{ \gb(\xb^1(t),\ldots,\xb^N(t))|\\
 \nonumber
  &\qquad\quad x^i_k(t)\sim\EuScript N(\mu_k^i(t),\sigma_i(t)), i\in[N], k\in[d]\}\\
 \nonumber
 = &- \nabla_{\bla}{{\tilde J^0_{N+1}(\bmu^1(t),\ldots,\bmu^N(t), \bla(t), \bsigma(t))}}.
 \end{align}


Using the notation $\be(t) =[\bmu(t),\bla(t)]$, we can rewrite the algorithm steps \eqref{eq:regpl}-\eqref{eq:dualpl} in the following form:
 \begin{align}
 \label{eq:pbavmu}
&\bmu^i(t+1) =\Proj_{\Ab}[\bmu^i(t) -\gamma_i(t+1)\sigma_i^2(t+1)\\
\nonumber
&\times\big(\Mb_i^0(\be(t)) +\Qb_i(\be(t),\bsigma(t))+\Rb_i(\be(t),\xb(t),\bsigma(t))\big)],\\
\label{eq:pbavlam}
&\bla(t+1) =\Proj_{\R^n_{+}}[\bla(t) -\beta_0(t+1)\times\big(-{\gb}(\bmu(t)) \\
\nonumber
&+\Qb_{N+1}(\be(t),\bsigma(t))+\Rb_{N+1}(\be(t),\xb(t),\bsigma(t))\big)],
\end{align}
where for $i\in[N]$
 \begin{align*}
&\Qb_i(\be(t),\bsigma(t)) =\tilde{\Mb}_i^0 (\be(t),\bsigma(t)) -\Mb_i^0(\be(t)),\cr
&\Rb_i(\xb(t),\be(t),\bsigma(t)) = \boldsymbol F_i(\xb(t),\be(t), \bsigma(t)) - \tilde{\Mb}_i^0 (\be(t),\sigma(t)), \cr
&\boldsymbol F_i(\xb(t),\be(t), \bsigma(t)) ={\hat J}^0_i(t)\frac{\xb^i(t) -\bmu^i(t)}{\sigma_i^2(t)} ,
\end{align*}
and $\tilde{\Mb}_i^0 (\cdot)=[\tilde M^0_{i,1}(\cdot), \ldots, \tilde M^0_{i,d}(\cdot)]^{\top}$
is the $d$-dimensional mapping with the following elements:
\begin{align}\label{eq:mapp2}
\tilde M^0_{i,k} (\be(t),\bsigma(t))=\frac{\partial {\tilde J^0_i(\be(t), \bsigma(t))}}{\partial \mu^i_k}, \mbox{ for $k\in[d]$}.
\end{align}
Furthermore,
 \begin{align*}
&\Qb_{N+1}(\be(t),\bsigma(t)) = \tilde {\boldsymbol M}^0_{N+1}(\be(t),\bsigma(t)) + {\gb}(\bmu(t)),\cr
&\Rb_{N+1}(\xb(t),\be(t),\bsigma(t))= -\hat \gb(t) - \tilde{\Mb}^0_{N+1} (\be(t),\bsigma(t)),
 \end{align*}
and $\tilde {\boldsymbol M}^0_{N+1}(\be(t),\bsigma(t))=\nabla_{\bla}{\tilde J^0_{N+1}(\be(t), \bsigma(t))}$.
The algorithm \eqref{eq:pbavmu}-\eqref{eq:pbavlam} falls under the framework of Robbins-Monro stochastic approximations procedure  \cite{Borkar}, where
$$\Mb^0(\be(t))=[\Mb_1^0(\be(t)), \ldots, \Mb_N^0(\be(t)), -{\gb}(\bmu(t))],$$
corresponds to the gradient term in stochastic approximation procedures. Furthermore,
 \begin{align*}
  \Qb(\be(t),\bsigma(t)) = [\Qb_1(\be(t),\bsigma(t)),\ldots,&\Qb_N(\be(t),\bsigma(t)),\cr
&\Qb_{N+1}(\be(t),\bsigma(t))],
 \end{align*}
is a disturbance of the gradient term, and
 \begin{align*}
  \Rb(\xb(t), \be(t),&\bsigma(t)) = [\Rb_1(\xb(t),\be(t),\bsigma(t)),\ldots,\cr
&\Rb_N(\xb(t), \be(t),\bsigma(t)),\Rb_{N+1}(\xb(t),\be(t),\bsigma(t))],
 \end{align*}
is a Martingale difference. Namely, according to \eqref{eq:mathexp} and \eqref{eq:mathexp1},
\begin{align}
\label{eq:mathexp2}
 &\Rb_i(\xb(t),\be(t),\bsigma(t)) = \boldsymbol F_i(\xb(t),\be(t), \bsigma(t)) \\
\nonumber
 &- \E_{\xb(t)}\{\boldsymbol F_i(\xb(t),\be(t), \bsigma(t))\},\; \;i\in[N], \\
\label{eq:mathexp3}
 & \Rb_{N+1}(\xb(t),\be(t),\bsigma(t))= -\hat \gb(t) + \E_{\xb(t)}\{\hat \gb(t)\}.
\end{align}

To ensure convergence of the iterates $\be(t)$, the step-sizes  $\beta_0(t)$, $\sigma_i(t)$, $\gamma_i(t)$, $i\in[N]$, need to satisfy certain assumptions. Let $\beta_i(t)=\gamma_i(t)\sigma_i^2(t)$, $\bemin(t)=\min_{i\in\{0, [N]\}}\beta_i(t)$, $\bemax(t)=\max_{i\in\{0, [N]\}}\beta_i(t)$, $\gamax(t)=\max_{i\in[N]}\gamma_i(t)$.

\begin{assumption}\label{assum:timestep}
The variance parameters $\sigma_i(t)$ and the step-size parameters $\beta_0(t)$, $\gamma_i(t)$, $i\in[N]$, are chosen such that

 1) $\sum_{t=0}^{\infty}\bemin(t)=\infty$,

 2) $\sum_{t=0}^{\infty}\bemax(t)-\bemin(t)<\infty,$

 3) $\sum_{t=0}^{\infty}\gamax^2(t)<\infty,$ $\sum_{t=0}^{\infty}\gamax(t)\sigmax^3(t) < \infty$.

 \end{assumption}
 \begin{rem}\label{rem:example}
 Similar to \cite{Kann_Shan} the agents choose algorithm parameters independently and the only coordination is in Assertion 2) above.
An example of sequences $\gamma_i(t)$, $\sigma_i(t)$, $i\in[N]$, $\beta_0(t)$ is the protocol for distributed optimization schemes  \cite{Kann_Shan}, where each regular agent picks  a positive integer $R_i$, $i\in [N]$, the dual player picks  a positive integer $N_0$, and  
  $\gamma_i(t)=\frac{1}{(t+R_i)^a}$, $\sigma_i(t)=\frac{1}{(t+R_i)^b}$, $i\in[N]$, $\beta_0(t)=\frac{1}{(t+N_0)^{a+2b}}$, with $a+2b\in(0.5,1]$, $2a>1$, and $a+3b>1$.
 \end{rem}

\mk{For our convergence results under coupling constraints, we will consider potential games as defined below.  }

\begin{assumption}\label{assum:potential}
 The game $\Gamma(N, \{A_i\}, \{J_i\}, C)$ admits a strictly convex potential function  $f:\R^{Nd} \rightarrow \R$, with $\frac{\partial f(\ab)}{\partial a^i_k} = M_{i,k}(\ab)$. This is equivalent to the Jacobian of the game mapping $\Mb:\R^{Nd}\to\R^{Nd}$ being symmetric  (Theorem 1.3.1 in \cite{FaccPang1}. )
 \end{assumption}

\begin{theorem}\label{th:main}
  Let Assumptions\r\ref{assum:convex}-\ref{assum:inftybeh}, \ref{assum:potential} hold in a game $\Gamma(N, \{A_i\}, \{J_i\}, C)$. 
  Let the regular players choose the states $\{\xb^i(t)\}$ at time $t$ according to the normal distribution $\EuScript N(\bmu^i(t),\sigma_i(t))$, where the mean parameters are updated as in \eqref{eq:regpl}. Let the action $\bla(t)$ of the dual player is updated according to \eqref{eq:dualpl}. Let the variance parameters and the step-size parameters satisfy Assumption\r\ref{assum:timestep} Then, as $t\to\infty$, the mean vector $\bmu(t)$ converges almost surely to the generalized Nash equilibrium $\bmu^*=\ab^*$ of the game $\Gamma$,  given any initial vector $[\bmu(0),\bla(0)]$, and the joint state $\xb(t)$ converges in probability to $\ab^*$.
\end{theorem}

 \begin{rem}
Analogously to optimization methods based on the gradient descent iterations,  condition 1) in Assumption\r\ref{assum:timestep}, $\sum_{t=0}^{\infty}\bemin(t) = \infty$, guarantees sufficient energy for the time-step parameter $\gamma_i(t)\sigma_i^2(t)$ to let the algorithm \eqref{eq:pbavmu}-\eqref{eq:pbavlam} get to a neighborhood of a desired stationary point for each $i$, whereas  condition $\sum_{t=0}^{\infty}\gamax^2(t) < \infty$ ensures the algorithm converges as time goes to infinity.
\end{rem}

\section{Analysis of the Algorithm Convergence}\label{sec:proof}
Our approach in proving Theorem\r\ref{th:main} is to first prove boundedness of the iterates $\be(t)$.  Next, we show that the limit of the iterates $\be(t)$ exists and satisfies the conditions for being a variational equilibrium of the game $\Gamma(N, \{A_i\}, \{J_i\}, C)$.
\subsection{Boundedness of the Algorithm Iterates}
First, we demonstrate that under conditions of Theorem\r\ref{th:main} the vector $\be(t)$ stays almost surely bounded for any $t\in\Z_+$. 
\begin{lem}\label{lem:boundedvec}
 Let Assumptions\r\ref{assum:convex}-\ref{assum:inftybeh} hold in a game $\Gamma(N, \{A_i\}, \{J_i\}, C)$ with coupled actions and $\be(t)=[\bmu(t),\bla(t)]$ be the vector updated in the run of the payoff-based algorithm \eqref{eq:pbavmu}-\eqref{eq:pbavlam}. Let the variance parameters and the step-size parameters satisfy Assumption\r\ref{assum:timestep}. Then, $\Pr\{\sup_{t\ge 0}\|\be(t)\|< \infty\}=1$.
\end{lem}
\begin{proof}
In the following, for simplicity in notation, we omit the argument $\bsigma(t)$ in the terms $\Mb^0$, $\tilde{\Mb}^0$, $\Qb$, and $\Rb$. In certain derivations, for the same reason we omit the time parameter $t$ as well. According to the vector form of the algorithm \eqref{eq:pbavmu}-\eqref{eq:pbavlam},
\begin{align}
\label{eq:pbavmu1}
\bmu^i(t+1) =&\Proj_{A_i}[\bmu^i(t) -\beta_i(t+1)(\Mb_i^0(\be(t))\\
\nonumber
&+\Qb_i(\be(t))+\Rb_i(\be(t),\xb(t)))],\\
\label{eq:pbavlam2}
\bla(t+1) =&\Proj_{\R^n_{+}}[\bla(t) -\beta_0(t+1)\times(-{\gb}(\bmu(t)) \\
\nonumber
&+\Qb_{N+1}(\be(t))+\Rb_{N+1}(\be(t),\xb(t)))].
\end{align}
Let $\be^*=[\bmu^*,\bla^*]\in \Ab\times\R^n_{+}$ be a Nash equilibrium of the associated game $\Gamma_{a}(\Ab\times\R^n_{+})$. This equilibrium exists and its norm is bounded, namely $\|\be^*\|<\infty$, according to Lemma\r\ref{th:exist_uncoupled}, Assertions 1) and 3).

Let us define the function $V(\be)=\|\be-\be^*\|^2$. We  consider the generating operator of the Markov process $\be(t)$
\begin{align*}
LV(\be)=E[V(\be(t+1))\mid \be(t)=\be]-V(t,\be).
\end{align*}
Our goal is to show that $LV(\be(t))$ satisfies sufficient decay, as per results on stability of discrete-time Markov processes (\cite{NH}, Theorem\r2.5.2). That is,  \[LV(\be(t))\le -\alpha(t+1)\psi(\be(t)) + f(t)(1+V(\be(t))),\]
   where the functions $\psi(.) \ge 0$, $f(.)>0$ will be identified as per requirements of (\cite{NH}, Theorem\r2.5.2) (repeated in Theorem\r\ref{th:finiteness} in the appendix for completeness).

Taking into account the iterative procedure for the update of $\be(t)$ above and the non-expansion property of the projection operator on a convex set, we get
\begin{align}\label{eq:nonexp}
 \|&\bmu_i(t+1)-\bmu_i^*\|^2= \|\Proj_{A_i}[\bmu_i(t)-\beta_i(t+1)(\Mb_i^0(\be(t))\cr
 & \qquad\qquad\qquad  +\Qb_i(\be(t))+\Rb_i(\xb(t),\be(t)))]-\bmu_i^*\|^2 \cr
 &\le \|\bmu_i(t)-\bmu_i^* -\beta_i(t+1)(\Mb_i^0(\be(t))\cr
 &\qquad\qquad\qquad  +\Qb_i(\be(t))+\Rb_i(\xb(t),\be(t)))\|^2\cr
 & = \|\bmu_i(t)-\bmu_i^*\|^2 - 2\beta_i(t+1)(\Mb^0(\be(t)), \bmu_i(t)-\bmu_i^*) \cr
 &\qquad-2\beta_i(t+1)(\Qb_i(\be(t))+\Rb_i(\xb(t),\be(t)), \bmu_i(t)-\bmu_i^*) \cr
 &\qquad + \beta_i^2(t+1)\|\Gb_i(\xb(t),\be(t))\|^2,
\end{align}
where, for ease of notation we have defined
\begin{align}\label{eq:G_1}
 \Gb_i(\xb(t),\be(t)) = \Mb_i^0(\be(t)) +\Qb_i(\be(t))+\Rb_i(\xb(t),\be(t)).
\end{align}
Similarly, we can bound the dual term
\begin{align}\label{eq:nonexp1}
 &\|\bla(t+1)-\bla^*\|^2\cr
 &= \|\Proj_{\R^n_{+}}[\bla(t) -\beta_0(t+1)(-{\gb}(\bmu(t))+\Qb_{N+1}(\be(t))\cr
 &\qquad\qquad\qquad\qquad\quad\qquad+\Rb_{N+1}(\be(t),\xb(t)))]-\bla^*\|^2 \cr
 &\le \|\bla(t)-\bla^* -\beta_{0}(t+1)(-{\gb}(\bmu(t))+\Qb_{N+1}(\be(t))\cr
 &\qquad\qquad\qquad\qquad\quad\qquad  +\Rb_{N+1}(\xb(t),\be(t)))\|^2\cr
 & = \|\bla(t)-\bla^*\|^2 - 2\beta_0(t+1)(-{\gb}(\bmu(t)), \bla(t)-\bla^*) \cr
 &-2\beta_0(t+1)(\Qb_{N+1}(\be(t))+\Rb_{N+1}(\xb(t),\be(t)), \bla(t)-\bla^*) \cr
 &\qquad\qquad\qquad + \beta_0^2(t+1)\|\Gb_{N+1}(\xb(t),\be(t))\|^2,
\end{align}
with
\begin{align}\label{eq:G_2}
 \Gb&_{N+1}(\xb(t),\be(t))\cr
 &= -{\gb}(\bmu(t)) +\Qb_{N+1}(\be(t))+\Rb_{N+1}(\xb(t),\be(t)).
\end{align}
Thus, taking into account the Martingale properties in \eqref{eq:mathexp2} and \eqref{eq:mathexp3} of the terms $\Rb_j$, $j\in[N+1]$, we obtain
\begin{align}\label{eq:LV}
LV(\be)&=\E[\|\be(t+1)-\be^*\|^2|\be(t)=\be]-\|\be-\be^*\|^2\cr
=&\sum_{i=1}^N\left(\E[\|\bmu_i(t+1)-\bmu_i^*\|^2|\be(t)=\be]-\|\bmu_i-\bmu_i^*\|^2\right)\cr
&\qquad+E[\|\bla(t+1)-\bla^*\|^2|\be(t)=\be]-\|\bla-\bla^*\|\cr
\le&-2\sum_{i=1}^N\beta_i(t+1)(\Mb_i^0(\be), \bmu_i-\bmu_i^*)\cr
&-2\beta_0(t+1)(-{\gb}(\bmu), \bla-\bla^*)\cr
&-2\sum_{i=1}^N\beta_i(t+1)(\Qb_i(\be), \bmu_i-\bmu_i^*)\cr
&-2\beta_0(t+1)(\Qb_{N+1}(\be), \bla-\bla^*)\cr
&+\sum_{i=1}^N\beta_i^2(t+1)\E\{\|\Gb_{i}(\xb(t),\be(t))\|^2|\be(t)=\be\}\cr
&+\beta_0^2(t+1)\E\{\|\Gb_{N+1}(\xb(t),\be(t))\|^2|\be(t)=\be\}.
\end{align}
Now, we bound the first two terms in the last expression above.
\begin{align}\label{eq:terms1}
& -2\sum_{i=1}^N\beta_i(t+1)(\Mb_i^0(\be), \bmu_i-\bmu_i^*)\cr
&\qquad\qquad\qquad-2\beta_0(t+1)(-{\gb}(\bmu), \bla-\bla^*)\cr
&=-2\bemin(t+1)(\Mb^0(\be),\be-\be^*)\cr
&\qquad\qquad\qquad+2\bemin(t+1)(\Mb^0(\be),\be-\be^*)\cr
&\qquad\qquad\qquad-2\sum_{i=1}^N\beta_i(t+1)(\Mb_i^0(\be), \bmu_i-\bmu_i^*)\cr
&\qquad\qquad\qquad-2\beta_0(t+1)(-{\gb}(\bmu), \bla-\bla^*)\cr
&\le-2\bemin(t+1)(\Mb^0(\be),\be-\be^*)\cr
&\qquad\qquad+2(\bemax(t+1)-\bemin(t+1))\|\Mb^0(\be)\|\|\be-\be^*\|\cr
&\le-2\bemin(t+1)(\Mb^0(\be),\be-\be^*)\cr
&\qquad\qquad+2(\bemax(t+1)-\bemin(t+1))k_1(1+V(\be)),
\end{align}
for some constant $k_1>0$, where the last inequality is due to the linear behavior of the mapping $\Mb^0(\be)$ at infinity (see Assumption\r\ref{assum:inftybeh}). Hence, \eqref{eq:LV} and \eqref{eq:terms1} imply
\begin{align}\label{eq:LV2}
LV(\be)&\le-2\bemin(t+1)(\Mb^0(\be),\be-\be^*)\cr
&+2(\bemax(t+1)-\bemin(t+1))k_1(1+V(\be))\cr
&+2\sum_{i=1}^N\beta_i(t+1)\|\Qb_i(\be)\| \|\bmu_i-\bmu_i^*\|\cr
&+2\beta_0(t+1)\|\Qb_{N+1}(\be)\| \|\bla-\bla^*\|\cr
&+\sum_{i=1}^N\beta_i^2(t+1)\E\{\|\Gb_{i}(\xb(t),\be(t))\|^2|\be(t)=\be\}\cr
&+\beta_0^2(t+1)\E\{\|\Gb_{N+1}(\xb(t),\be(t))\|^2|\be(t)=\be\}.\cr
\end{align}

Let us analyze the terms containing $\Qb_i$ for $i \in [N+1]$  in \eqref{eq:LV2}. First, we will show that the mapping $\tilde{\Mb}_i^0 (\be(t))$ (see \eqref{eq:mapp2}) evaluated at $\be(t)$ is equivalent to the extended game mapping (see Definition\r\ref{def:mapping_ext}) in mixed strategies, that is, for $i \in [N+1]$
\begin{align}\label{eq:gradmix}
 \tilde{\Mb}_i^0 (&\be(t))=\int_{\mathbb R^{Nd}}{\Mb_i^0} (\bx,\bla)p(\bmu(t),\bx)d\bx.
\end{align}
Indeed, using the notations
\begin{align*}
&\mu^i_{-k}=(\mu^i_1,\ldots,\mu^{i}_{k-1},\mu^{i}_{k-1},\ldots\mu^{i}_{d})\in\R^{d-1},\\
&x^i_{-k}=(x^i_1,\ldots,x^{i}_{k-1},x^{i}_{k-1},\ldots x^{i}_{d})\in\R^{d-1},\\
&p(\mu^i_{-k},x^i_{-k})=\frac{1}{(\sqrt{2\pi}\sigma_i)^{d-1}}\exp\left\{-\sum_{j\ne k}\frac{(x^i_j-\mu^i_j)^2}{2\sigma_i^2}\right\}\\
&p(\bmu^{-i},\bx^{-i})=\prod_{j\ne i, j=1}^N\frac{1}{(\sqrt{2\pi}\sigma_j)^{d}}\exp\left\{-\sum_{k=1}^d\frac{(x^j_k-\mu^j_k)^2}{2\sigma_j^2}\right\},
\end{align*}
we can show that for any $i\in[N]$, $k\in[d]$, $\tilde{M^0}_{i,k}(\be)$
\begin{align}\label{eq:diffpotfun}
\allowdisplaybreaks
&\tilde{M^0}_{i,k}(\be)=\tilde{M^0}_{i,k}(\bmu,\bla) \cr
&= \frac{1}{\sigma_i^2}\int_{\R^{Nd}}J^0_i(\bx,\bla)(x^i_k - \mu^i_k)p(\bmu,\bx)d\bx\cr
&= -\int_{\R^{Nd}}J^0_i(\bx,\bla) p(\mu^i_{-k},x^i_{-k}) p(\bmu^{-i},\bx^{-i})\frac{1}{\sqrt{2\pi}\sigma_i} \cr
&\qquad\qquad\qquad\qquad\qquad\qquad\qquad \times d\left(e^{-\frac{(x^i_k-\mu^i_k)^2}{2\sigma_i^2}}\right) d\bx^{-i}\cr
 &=-\int_{\R^{Nd-1}}\left(J^0_i(\bx,\bla)e^{-\frac{(x_k^i-\mu_k^i)^2}{2\sigma_i^2}}\right)\bigg|_{-\infty(x_k^i)}^{\infty(x_k^i)}\cr
&\qquad\qquad\qquad \times p(\mu^i_{-k},x^i_{-k}) p(\bmu^{-i},\bx^{-i})\frac{1}{\sqrt{2\pi}\sigma_i}d\bx^{-i}\cr
 &\qquad\qquad\qquad\qquad+\int_{\R^{Nd}}\frac{\partial J^0_i(\bx,\bla)}{\partial x^i_k}p(\bmu,\bx)d\bx\cr
 &=\int_{\R^{Nd}}\frac{\partial J^0_i(\bx,\bla)}{\partial x^i_k}p(\bmu,\bx)d\bx.
 \end{align}
The above holds since according to Assumption~\ref{assum:inftybeh},
\begin{align*}
\lim_{x_k^i\to\infty(-\infty)}J^0_i(\bx,\bla)e^{-\frac{(x_k^i-\mu_k^i)^2}{2\sigma_i^2}}=0,
\end{align*}
 for any fixed $\mu_k^i$, $\bx^{-i}$, and $\bla$.
 Thus, \eqref{eq:gradmix} holds for each regular player $i \in [N]$.
 Moreover, for the dual player
 \begin{align}\label{eq:diffpotfun1}
\tilde {\boldsymbol M}^0_{N+1}(\be)&=\int_{\R^{Nd}}\nabla_{\bla}{J^0_{N+1}(\bx,\bla)}p(\bmu,\bx)d\bx\cr
&=-\int_{\R^{Nd}}\gb(\bx)p(\bmu,\bx)d\bx.
 \end{align}
Since $\Qb_i(\be(t))=\tilde{\Mb}_i^0(\be(t)) - \Mb_i^0(\be(t))$ and due to Assumption~\ref{assum:CG_grad} and equation \eqref{eq:gradmix}, we obtain the following:
\begin{align}\label{eq:Qterm1}
 \|\Qb_i(\be)\|&=\|\int_{\R^{Nd}}[\Mb_i^0(\bx,\bla)-\Mb_i^0(\bmu,\bla)]p(\bmu,\bx)d\bx\| \cr
 &\le \int_{\R^{Nd}}\|\Mb_i^0(\bx,\bla) - \Mb_i^0(\bmu,\bla)\| p(\bmu,\bx) d\bx\cr
 &\le \int_{\R^{Nd}} L_i(\bla) \|\bx - \bmu\| p(\bmu,\bx) d\bx \cr
 \nonumber &\le \int_{\R^{Nd}} L_i(\bla) \left(\sum_{i=1}^{N}\sum_{k=1}^{d}|x^i_k - \mu^i_k|\right) p(\bmu,\bx) d\bx\\
 &= O(\sum_{i=1}^{N}\sigma_i)(1+\|\be-\be^*\|),
 \end{align}
where the last equality is due to the fact that the first central absolute moment of a random variable with a normal distribution $\EuScript N(\mu,\sigma)$ is $O(\sigma)$ and $L_i(\bla)$ is a linear function of $\bla$ (see Assumption~\ref{assum:CG_grad}) and, hence, $L_i(\bla)\le k(1+\|\be-\be^*\|)$ for some constant $k$.
Thus,
\begin{align}\label{eq:Qterm2}
\|\Qb_i(\be)\|\|\bmu_i-\bmu_i^*\|\le O(\sum_{i=1}^{N}\sigma_i)(1+V(\be)).
\end{align}

Similarly, using Assumption~\ref{assum:CG_grad} and equality\r\eqref{eq:diffpotfun1}, we have
\begin{align}\label{eq:Qterm4}
 & \|\Qb_{N+1}(\be)\| = \|\tilde {\boldsymbol M}^0_{N+1}(\be) + {\gb}(\bmu)\| \le O(\sum_{i=1}^N\sigma_i),\cr
& \|\Qb_{N+1}(\be)\|\|\bla-\bla^*\|\le O(\sum_{i=1}^N\sigma_i)(1+V(\be)).
\end{align}

Finally, we bound the last two terms in \eqref{eq:LV2}.
Since $\E(\xi-\E\xi)^2\le\E\xi^2$ and taking into account \eqref{eq:mathexp2}, we have \begin{align}\label{eq:Rineq1}
  &\E\{\|\Rb_i(\xb(t),\be(t))\|^2|\be(t)=\be\}\cr
  &\le\sum_{k=1}^{d}\int_{\mathbb R^{Nd}}{J_i^0}^2(\bx,\bla)\frac{(x^i_k - \mu^i_k)^2}{\sigma_i^4(t)} p(\bmu,\bx)d\bx.
  \end{align}
Thus, we can use Assumption~\ref{assum:inftybeh}, Remark\r\ref{rem:fun_g}, and the fact that $J_i^0(\bx,\bla)$ is affine in $\bla$ to get the next inequality:
 \begin{align}\label{eq:Rineq}
  \E\{\|\Rb_i(\xb(t),&\be(t))\|^2|\be(t)=\be\}\cr
  &\le f(\bmu, \bsigma(t))\left(\frac{1}{\sigma_i^4(t)}+k_2 V(\be)\right),
 \end{align}
 where $f(\bmu, \bsigma(t))$ is a polynomial of $\mu_i$ and $\sigma_i(t)$, $i\in[N]$, and $k_2$ is some positive constant.

Furthermore, taking into account boundedness of $\bmu(t)$ and affinity of the mapping $\Mb^0(\be)$ with respect to $\bla$, we obtain the following bound for any $i\in[N]$:
\begin{align}\label{eq:Gterm1}
&\beta_i^2(t+1)\E\{\|\Gb_i(\xb(t),\be(t))\|^2|\be(t)=\be\}\cr
&\le\beta_i^2(t+1)(\|\Mb_i^0(\be)\|^2 + \|\Qb_i(\be(t))\|^2) \cr
&\qquad\qquad+2\beta_i^2(t+1)\|\Mb_i^0(\be)\|\|\Qb_i(\be)\|\cr
&\qquad\qquad+\beta_i^2(t+1)(\E\{\|\Rb_i(\xb(t),\be(t))\|^2|\be(t)=\be\})\cr
&\le\beta_i^2(t+1)(k_3+k_4\|\bla\|^2+O(\sigma_i^2(t))(1+\|\be-\be^*\|)^2)\cr
&\quad\quad\quad+2\beta_i^2(t+1)(k_5+k_6\|\bla\|)O(\sigma_i(t)(1+\|\be-\be^*\|))\cr
&\qquad\qquad\qquad\quad+O(\gamma_i^2(t+1))+O(\beta_i^2(t+1))V(\be)\cr
&\le O(\beta_i^2(t+1)(1+\sigma_i^2(t)+\sigma_i(t))+\gamma_i^2(t+1))\cr
&\qquad\qquad\qquad\qquad\quad\qquad\qquad\qquad\times(1+V(\be)),
\end{align}
where $k_j$, $j=3,\ldots,6$, are some positive constants.

For the term $\E\{\|\Gb_{N+1}(\xb(t),\be(t))\|^2|\be(t)=\be\}$ we can first derive the following bound: \begin{align*}
  \E\{\|&\Rb_{N+1}(\xb(t),\be(t))\|^2|\be(t)=\be\}\cr
  &\le\int_{\mathbb R^{Nd}}\|\gb(\bx)\|^2 p(\bmu,\bx)d\bx=\phi(\bmu,\bsigma(t)),
 \end{align*}
where $\phi(\bmu, \bsigma(t))$ is a polynomial of $\mu_i$ and $\sigma_i(t)$, $i\in[N]$ (see Remark\r\ref{rem:fun_g}). Hence, according to boundedness of $\bmu(t)$ and the fact that $\sigma(t)$ goes to zero, we obtain
\begin{align}\label{eq:Gterm2}
&\beta_0^2(t+1)\E\{\|\Gb_{N+1}(\be(t))\|^2|\be(t)=\be\}\cr
&\le\beta_0^2(t+1)(\|\gb(\bmu)\|^2 + \|\Qb_{N+1}(\be(t))\|^2) \cr
&\qquad\qquad+2\beta_0^2(t+1)\|\gb(\bmu)\|\|\Qb_{N+1}(\be)\|\cr
&\qquad\qquad+\beta_0^2(t+1)(\E\{\|\Rb_{N+1}(\be(t))\|^2|\be(t)=\be\})\cr
&\le \beta_0^2(t+1)(k_7+O(\sigmax^2(t)) + O(\sigmax(t))),
\end{align}
where $k_7$ is some positive constant.

Since $\be^*$ is a NE in $\Gamma_{a}(\Ab\times\R^n_{+})$, Assertion 1) in Theorem\r\ref{th:exist_uncoupled} implies that
\[(\Mb^0(\be^*),\be-\be^*)\ge 0,\]
for any $\be\in\Ab\times\R^n_{+}$.
According to pseudo-monotonicity of $\Mb^0$ in Assumption\r\ref{assum:CG_grad}, the inequality above implies
\begin{align}\label{eq:function1}
 (\Mb^0(\be),\be-\be^*)\ge 0\mbox{ for any $\be\in\Ab\times\R^n_{+}$}.
\end{align}
Thus, bringing \eqref{eq:LV2}, \eqref{eq:Qterm2}, \eqref{eq:Qterm4}, \eqref{eq:Gterm1}, and \eqref{eq:Gterm2} together, we get
\begin{align}\label{eq:LV1}
LV(\be)\le&- 2\bemin(t+1)(\Mb^0(\be), \be-\be^*)\cr
&+p(t)(1+V(\be)),\cr
p(t)=&O(\bemax(t+1)-\bemin(t+1))\cr
&+O(\bemax(t+1)\sigmax(t)+\gamax^2(t+1)).
\end{align}
Hence, using conditions on parameters in Assumption\r\ref{assum:timestep}, we get $\sum_{t=0}^{\infty}p(t)<\infty$. Finally, taking into account \eqref{eq:function1}, \eqref{eq:LV1}, $\sum_{t=0}^{\infty}\bemin(t)=\infty$, and Theorem\r\ref{th:finiteness}, we conclude that $\be(t)$ is finite almost surely for any $t\in\Z_{+}$ during the run of the algorithm irrespective of $\be(0)$.
\end{proof}

\subsection{Convergence of the Algorithm Iterates}
Having established boundedness of the iterates in the algorithm \eqref{eq:regpl}-\eqref{eq:dualpl}, in this subsection we prove Theorem~\ref{th:main}. 
\begin{proof}[Proof of Theorem\r\ref{th:main}]
Under Assumption~\ref{assum:potential}, the game can be reformulated as a constrained  optimization problem. Indeed, if the game $\Gamma$ is potential with a strictly convex potential function $f(\ab)$, then 
$\Mb(\ab) = \nabla f(\ab)$ and the problem of finding a variational equilibrium  is equivalent to solving 
\begin{align}\label{eq:problem}
 \mbox{minimize} \mbox{ }  & f(\ab) \cr
 \mbox{subject to} \mbox{ } & g_j(\ab) \leq 0, \quad j=1,\ldots,m \cr
                   & \ab\in \Ab = A_1\times\ldots\times A_N\subseteq \R^{Nn}.
\end{align}
This equivalence follows from the fact that, due to the definition of potential games, the unique minimum solving the problem above is a variational  Nash equilibrium in the game, whereas strict monotonicity of the game mapping $\Mb$ implies uniqueness of the solution of $VI(\EQ,\Mb)$ (see, for example, \cite{FaccPang1}) and, thus, uniqueness of the variational generalized Nash equilibrium. We call the above the primal problem. 

Consider the Lagrangian function for the primal problem 
\begin{align}\label{eq:Lagr}
 \Lag(\ab,\bla) = f(\ab) + (\bla,\gb(\ab)).
\end{align}
Then the dual function is defined as $\sup_{\bla\in\R_+^{n}}\inf_{\ab\in \Ab}\Lag(\ab,\bla)$. 
Given Assumption~\ref{assum:potential}, we conclude that for any primal-dual optimal pair $(\ab^*, \bla^*)$ the vector $\ab^*$ is the unique (due to strict monotonicity) variational Nash equilibrium in the game $\Gamma$. Under Assumption~\ref{assum:Slaters} the Karush-Kuhn-Tucker (KKT) \cite{KKT} conditions hold and consequently,
 $(\ab^*, \bla^*)$ is  a primal-dual optimal pair if and only if:
 
 (1) $\ab^*$ is primal feasible and $\bla^*\in\R_+^{n}$ ($\bla^*$ is dual feasible);
 
 (2) $\ab^*$ attains the minimum in $\inf_{\ab\in\Ab}\Lag(\ab,\bla^*)$; 
 
 (3) $\bla^*$ attains the maximum in $\sup_{\bla\in\R_+^{n}}\Lag(\ab^*,\bla)$.
Using the above characterization with the bounds derived in the previous section, we can establish convergence of the algorithm.  First, from Inequality  \eqref{eq:LV1} we obtain that for any Nash equilibrium $\be^*$ 
\begin{align}\label{eq:dist1}
\E&\{\|\be(t+1)-\be^*\|^2|\EuScript F_t\}\le\|\be(t)  - \be^*\|^2 \cr
-&2 \bemin(t)(\Mb^0(\be(t)),\be(t)  - \be^*) +h(t),
\end{align}
where $\EuScript F_t$ is the $\sigma$-algebra generated by the random variables $\{\be(k), k \le t\}$ and $h(t) = p(t)(1+V(\be(t)))$. Due to Lemma~\ref{lem:boundedvec}, $\be(t)$ is bounded almost surely and, thus, according to the estimations for $p(t)$ in the proof of Lemma~\ref{lem:boundedvec}, $\sum_{t=0}^{\infty}h(t)  < \infty$. Second, we will bound the term $(\Mb^0(\be(t)),\be(t)  - \be^*) $ using the Lagrangian of the game  $\Lag(\ab,\bla)$. 

Due to definition of the mapping $\Mb^0$ and Assumption~\ref{assum:potential}, 
\begin{align*}
 (\Mb^0(\be(t)),\be(t)  - \be^*) = &(\nabla_{\bmu}\Lag(\bmu(t),\bla(t)),\bmu(t)-\bmu^*) \cr
 & -(\nabla_{\bla}\Lag(\bmu(t),\bla(t)),\bla(t)-\bla^*),
\end{align*}
where $\Lag(\bmu,\bla)$ is the Lagrangian function defined in \eqref{eq:Lagr}.
Due to convexity of $f$ and $g_1,\ldots, g_m$, according to KKT optimal conditions above, we get for any $\bmu\in \Ab$ and $\bla\in\R^n_{+}$
\begin{align}\label{eq:ineq1}
\Lag(\bmu,\bla^*)-\Lag(\bmu^*,\bla^*)\ge 0,
\end{align}
 \begin{align}\label{eq:ineq2}
 \Lag(\bmu^*,\bla^*)-\Lag(\bmu^*,\bla)\ge 0.
\end{align}
Due to convexity of $\Lag(\bmu,\bla)$ in $\bmu$ for any $\bla\in\R^n_{+}$ 
 \begin{align}\label{eq:dot1}
(&\nabla_{\bmu}\Lag(\bmu(t),\bla(t)),\bmu(t)-\bmu^*)  \ge \Lag(\bmu(t),\bla(t))-\Lag(\bmu^*,\bla(t))
\end{align}
and due to linearity of $\Lag(\bmu,\bla)$ in $\bla$, we obtain
\begin{align}\label{eq:dot2}
 (&\nabla_{\bla}\Lag(\bmu(t),\bla(t)),\bla(t)-\bla^*) = \Lag(\bmu(t),\bla(t))-\Lag(\bmu(t),\bla^*).
 \end{align}
Bringing \eqref{eq:dot1} and \eqref{eq:dot2} together and taking into account \eqref{eq:ineq1}, \eqref{eq:ineq2}, adding and subtracting $\Lag(\bmu^*,\bla^*)$ we conclude that
\begin{align}
\label{eq:dist2}
\E&\{\|\be(t+1)-\be^*\|^2|\EuScript F_t\}
\le\|\be(t)  - \be^*\|^2 -2 \bemin(t)\\
\nonumber
&\times[\Lag(\bmu(t),\bla^*)-\Lag(\bmu^*,\bla^*) + \Lag(\bmu^*,\bla^*)-\Lag(\bmu^*,\bla(t))] \cr
&\qquad\qquad\qquad+h(t).
\end{align}
Thus, we can apply the result of Robbins and Siegmund (Lemma~\ref{th:th_nonnegrv}) to the inequality \eqref{eq:dist2} to conclude that almost surely

1) $\|\be(t+1)-\be^*\|$ converges,

2) $\sum_{t=0}^{\infty}\bemin(t)(\Lag(\bmu(t),\bla^*)-\Lag(\bmu^*,\bla^*) + \Lag(\bmu^*,\bla^*)-\Lag(\bmu^*,\bla(t)))<\infty.$

As $\sum_{t=0}^{\infty}\bemin(t)=\infty$, 2) implies that there exists a  subsequence $\be(t_l) =(\bmu(t_l),\bla(t_l))$ such that almost surely
\begin{align*}
 \lim_{l\to\infty}\{[&\Lag(\bmu(t_l),\bla^*)-\Lag(\bmu^*,\bla^*)] \cr
 +& [\Lag(\bmu^*,\bla^*) - \Lag(\bmu^*,\bla(t_l))]\}=0.
\end{align*}
Since $\Lag(\bmu(t_l),\bla^*)-\Lag(\bmu^*,\bla^*)\ge0$ and $\Lag(\bmu^*,\bla^*) - \Lag(\bmu^*,\bla(t_l))\ge 0$ for any $l$, the above holds if and only if
\begin{align*}
\lim_{l\to\infty}\Lag(\bmu(t_l),\bla^*)=\Lag(\bmu^*,\bla^*) \\
\lim_{l\to\infty}\Lag(\bmu^*,\bla(t_l)) = \Lag(\bmu^*,\bla^*).
\end{align*}
As a subsequence of $\{\be(t)\}$, the sequence $\be(t_l) =(\bmu(t_l),\bla(t_l))$ is bounded almost surely. Hence, we can choose an almost surely convergent subsequence $\be(t_{l_s}) =(\bmu(t_{l_s}),\bla(t_{l_s}))$ such that $\lim_{s\to\infty}\be(t_{l_s}) = \hat \be = (\hat{\bmu},\hat{\bla})$  with probability 1, where $\hat{\bmu}\in\Ab$ and  $\hat{\bla}\in\R^n_{+}$ (since $\Ab$ and $\{\bla: \, \bla\in\R^n_{+}\}$ are closed). Due to the last two equalities above, almost surely
\begin{align*}
&\lim_{s\to\infty}\Lag(\bmu(t_{l_s}),\bla^*) = \Lag(\hat{\bmu},\bla^*)= \Lag(\bmu^*,\bla^*)\\
&\lim_{s\to\infty}\Lag(\bmu^*,\bla(t_{l_s}))= \Lag(\bmu^*,\hat{\bla}) = \Lag(\bmu^*,\bla^*).
\end{align*}
Since $f$ is strictly convex and $g_i$, $i=1, \dots, m$ are convex over $\Ab$, the equality  $\Lag(\hat{\bmu},\bla^*)= \Lag(\bmu^*,\bla^*) = \min_{\bmu\in \Ab} \Lag(\bmu,\bla^*)$ implies that $\hat \bmu = \bmu^*$. 
Moreover, due to dual feasibility of $\hat \bla$, the equality $\hat \bmu = \bmu^*$ together with the equality  $\Lag(\bmu^*,\hat{\bla}) = \Lag(\bmu^*,\bla^*) = \max_{\bla\in\R^n_{+}}\Lag(\bmu^*,\bla)$ imply that  $\hat \be = (\bmu^*,\hat{\bla})$ is an optimal primal dual pair and hence, a Nash equilibrium of the game. Since assertion 1) above holds for any Nash equilibrium $\be^*$, replacing $\be^*$ by $\hat {\be}$ in \eqref{eq:dist2}, we conclude that $\|\be(t+1)-\hat \be\|$  converges. Since there exists a subsequence $\be(t_{l_s})$ which converges to $\hat \be$, we get that almost surely
$\lim_{t\to\infty}\be(t) = \hat\be.$
Hence, $\Pr\{\lim_{t\to\infty}\bmu(t)= \bmu^*\}=1$.

Finally,
Assumption\r\ref{assum:timestep} implies that $\lim_{t\to\infty}\sigma_i(t)=0$ for all $i\in[N]$.
Taking into account that $\xb(t)\sim\EuScript N(\bmu(t),\bsigma(t))$, we conclude that $\xb(t)$ converges weakly to a Nash equilibrium $\ab^*=\bmu^*$ as time runs. Moreover, according to  Portmanteau Lemma \cite{portlem}, this convergence is also in probability.
\end{proof}
\section{Convergence with only local constraints}\label{sec:rate}
\subsection{Convergence to Nash equilibria under relaxed conditions}
We consider the game without coupling constraints, namely $C = \R^{Nd}$ (or equivalently $\gb \equiv 0$), and relax the assumption 
existence of a potential function in the game $\Gamma$. In particular, Assumptions~\ref{assum:convex} and \ref{assum:potential} are replaced by the following one.
 \begin{assumption}\label{assum:convex1}
 For all $i\in[N]$ the set $A_i$ is convex and compact, the cost function $J_i(\ab^i, \ab^{-i})$ is defined on $\R^{Nd}$, continuously differentiable in $\ab$ and the game mapping $\Mb$ is strictly monotone.
\end{assumption}

In this case, as $C = \R^{Nd}$, the payoff-based procedure \eqref{eq:regpl}-\eqref{eq:dualpl} is modified as follows: 
\begin{align}
\label{eq:nocoupling}
 &\bmu^i (t+1)=\\
 \nonumber
 & \Proj_{A_i}\left[\bmu^i(t) -\gamma_i(t+1)\sigma_i^2(t+1){\hat J_i(t)}\frac{{\xb^i(t)} -\bmu^i(t)}{\sigma_i^2(t)} \right],
 \end{align}
where $i\in[N]$, $\gamma_i(t+1)$ is a step-size parameter chosen by  player $i$ and $\hat J_i(t) = J_i(\xb(t))$ is the current observation of the $i$th player's cost function in the game $\Gamma$.

We make the following assumption regarding parameters $\gamma_i(t)$, $\sigma_i(t)$, and $\beta_i(t) = \gamma_i(t)\sigma^2_i(t)$ in the procedure \eqref{eq:nocoupling}.
\begin{assumption}\label{assum:timestep1}
The variance parameters $\sigma_i(t)$ and the step-size parameters $\gamma_i(t)$, $i\in[N]$, are chosen such that

 1) $\sum_{t=0}^{\infty}\bemin(t)=\infty$,

 2) $\sum_{t=0}^{\infty}\bemax(t)-\bemin(t)<\infty,$

 3) $\sum_{t=0}^{\infty}\gamax^2(t)\sigmax^2(t) < \infty$, $\sum_{t=0}^{\infty}\gamax(t)\sigmax^3(t) < \infty$.

 \end{assumption}
 
\begin{theorem}\label{th:main2}
 Let Assumptions\r\ref{assum:CG_grad}-\ref{assum:inftybeh}, \ref{assum:convex1}, and \ref{assum:timestep1} hold in a game $\Gamma(N, \{A_i\}, \{J_i\}, \R^{nN})$. 
  Let the players choose the states $\{\xb^i(t)\}$ at time $t$ according to the normal distribution $\EuScript N(\bmu^i(t),\sigma_i(t))$, where the mean parameters are updated as in \eqref{eq:nocoupling}.  Then, as $t\to\infty$, the mean vector $\bmu(t)$ converges almost surely to the Nash equilibrium $\bmu^*=\ab^*$ of the game $\Gamma$,  given any initial vector $\bmu(0)$, and the joint state $\xb(t)$ converges in probability to $\ab^*$.
\end{theorem}

\begin{proof}
 As before, let $V(\bmu) = \|\bmu - \bmu^*\|^2$, where $\bmu^*$ is the unique Nash equilibrium in the game $\Gamma$ (existence and uniqueness of $\bmu^*$ follows from Proposition 2.3.3 in \cite{FaccPang1}).
 Following the discussion in the proof of Lemma~\ref{lem:boundedvec}, we can rewrite the inequality \eqref{eq:LV2} as follows
 \begin{align}\label{eq:LV_noncoupl}
LV&(\bmu)\le-2\bemin(t+1)(\Mb(\bmu),\bmu-\bmu^*)\\
\nonumber
&+2(\bemax(t+1)-\bemin(t+1))k_1(1+V(\bmu))\\
\nonumber
&+2\sum_{i=1}^N\beta_i(t+1)\|\Qb^0_i(\bmu)\| \|\bmu_i-\bmu_i^*\|\\
\nonumber
&+\sum_{i=1}^N\beta_i^2(t+1)\E\{\|\Gb^0_{i}(\xb(t),\bmu(t))\|^2|\bmu(t)=\bmu\}, \text{where}\\
\label{eq:G_noncoupl}
& \Gb^0_i(\xb(t),\bmu(t)) = \Mb_i(\bmu(t)) +\Qb^0_i(\bmu(t))+\Rb^0_i(\xb(t),\bmu(t)),
\end{align}
and $\Qb^0_i(\bmu(t))$, $\Rb^0_i(\xb(t),\bmu(t))$ are obtain from $\Qb_i(\be(t))$, $\Rb_i(\xb(t),\be(t))$ by letting $\gb \equiv 0$.
Similarly to \eqref{eq:Qterm1} and \eqref{eq:Qterm2},
\begin{align}\label{eq:Qterm1_uncoupl}
 \|\Qb^0_i(\bmu)\|&= O(\sum_{i=1}^{N}\sigma_i)\\
\label{eq:Qterm2_uncoupl}
\|\Qb^0_i(\bmu)\|\|\bmu_i-\bmu_i^*\|&\le O(\sum_{i=1}^{N}\sigma_i)(1+V(\bmu)).
\end{align}
Moreover, similarly to \eqref{eq:Rineq1}, we conclude that
\begin{align}\label{eq:Rineq1_uncouple}
  &\E\{\|\Rb_i(\xb(t),\bmu(t))\|^2|\bmu(t)=\bmu\}\cr
  &\le\sum_{k=1}^{d}\int_{\mathbb R^{Nd}}{J_i}^2(\bx)\frac{(x^i_k - \mu^i_k)^2}{\sigma_i^4(t)} p(\bmu,\bx)d\bx.
  \end{align}
Thus, we can use Assumption~\ref{assum:inftybeh} to get
\begin{align}\label{eq:Rineq_uncoupl}
  \E\{\|\Rb_i(\xb(t),&\bmu(t))\|^2|\bmu(t)=\bmu\}\le \frac{f(\bmu, \bsigma(t))}{\sigma_i^2(t)},
 \end{align}
 where $f(\bmu, \bsigma(t))$ is a quadratic function of $\mu_i$ and polynomial in $\sigma_i(t)$, $i\in[N]$.
 Taking into account \eqref{eq:G_noncoupl}-\eqref{eq:Rineq_uncoupl} and Assumption~\ref{assum:inftybeh}, we conclude that for some positive constants $k_2$, $k_3$
 \begin{align}\label{eq:Gterm1_uncoupl}
&\beta_i^2(t+1)\E\{\|\Gb_i(\xb(t),\bmu(t))\|^2|\bmu(t)=\bmu\}\cr
&\le\beta_i^2(t+1)(\|\Mb_i(\bmu)\|^2 + \|\Qb^0_i(\bmu(t))\|^2) \cr
&\qquad\qquad+2\beta_i^2(t+1)\|\Mb_i(\bmu)\|\|\Qb^0_i(\bmu)\|\cr
&\qquad\qquad+\beta_i^2(t+1)(\E\{\|\Rb_i(\xb(t),\bmu(t))\|^2|\bmu(t)=\bmu\})\cr
&\le\beta_i^2(t+1)(k_2+O(\sigma_i^2(t))(1+\|\bmu-\bmu^*\|)^2)\cr
&\quad\quad\quad+2\beta_i^2(t+1)k_3(1+\|\bmu-\bmu^*\|)O(\sigma_i(t))\cr
&\qquad\qquad\qquad\quad+O(\gamma_i^2(t+1)\sigma_i^2(t+1))(1+V(\bmu))\cr
&\le O(\beta_i^2(t+1)+\gamma_i^2(t+1)\sigma_i^2(t+1))(1+V(\bmu)).
\end{align}
Next, bringing estimations \eqref{eq:Qterm2_uncoupl} and \eqref{eq:Gterm1_uncoupl} into \eqref{eq:LV_noncoupl}, we obtain 
 \begin{align}\label{eq:LVf_noncoupl}
LV(\bmu)\le&-2\bemin(t+1)(\Mb(\bmu),\bmu-\bmu^*) \cr
&+ O(s(t))(1+V(\bmu)),
\end{align}
where $s(t) = \bemax(t)-\bemin(t) + \gamma_{\max}(t)\sigma_{\max}^3(t)+\gamma_{\max}^2(t)\sigma_{\max}^2(t)$. Taking into account that $\Ab$ is compact, we conclude that
 \begin{align}\label{eq:f_noncoupl}
\E\{\|\bmu(t+1)-\bmu^*\|^2&|\EuScript F_t\}\le\|\bmu(t)  - \bmu^*\|^2 \cr
&-2\bemin(t+1)(\Mb(\bmu(t)),\bmu(t)-\bmu^*) \cr
&+ O(s(t)).
\end{align}
Moreover, for any $t$ we have $(\Mb(\bmu(t)),\bmu(t)-\bmu^*)>(\Mb(\bmu^*)),\bmu(t)-\bmu^*)\ge0$, due to \mk{pseudo-}monotonicity of $\Mb$ and the fact that $\bmu^*$ is the Nash equilibrium.
Thus, using Theorem~\ref{th:th_nonnegrv} we conclude that almost surely 

1) $\|\bmu(t+1)-\bmu^*\|$ converges,

2) $\sum_{t=0}^{\infty}\bemin(t)(\Mb(\bmu(t)),\bmu(t)-\bmu^*)<\infty.$

As $\sum_{t=0}^{\infty}\bemin(t)=\infty$, 2) implies that there exists a  subsequence $\bmu(t_l)$ such that almost surely
\begin{align}\label{eq:sub}
 \lim_{l\to\infty}(\Mb(\bmu(t_l)),\bmu(t_l)-\bmu^*)=0.
\end{align}
Since $\bmu(t_l)$ is bounded almost surely for any $l$, we can choose a convergent subsequence $\bmu(t_{l_s})$ such that 
$\lim_{s\to\infty}\bmu(t_{l_s}) = \bmu'$
for some $\bmu'$. Hence, due to \eqref{eq:sub},
\begin{align}
\label{eq:coco}
(\Mb(\bmu'),\bmu'-\bmu^*)=0,
\end{align}
which together with strict monotonicity of $\Mb$ implies $\bmu'=\bmu^*$. 
Thus, as $\|\bmu(t+1)-\bmu^*\|$ converges almost surely and there exists a subsequence $\bmu(t_{l_s})$ which converges to $\bmu^*$ almost surely, we get that $\Pr\{\lim_{t\to\infty}\bmu(t)= \bmu^*\}=1.$
Finally,
Assumption\r\ref{assum:timestep} implies that $\lim_{t\to\infty}\sigma_i(t)=0$ for all $i\in[N]$.
Taking into account that $\xb(t)\sim\EuScript N(\bmu(t),\bsigma(t))$, we conclude that $\xb(t)$ converges weakly to a Nash equilibrium $\ab^*=\bmu^*$ as time runs. Moreover, according to  Portmanteau Lemma \cite{portlem}, this convergence is also in probability.
\end{proof}

\begin{rem} \mk{In a prior work \cite{Tat_ifac2017}, we  showed convergence to Nash equilibria under pseudo-monotonicity of the game mapping, leveraging the proof technique in \cite{ZhuFrazzoli}. Recently,  in \cite{grammatico2018comments} it was  shown with a counterexample  that pseudo-monotonicity is not sufficient for the convergence results in \cite{ZhuFrazzoli}. This implies that our payoff-based algorithm also would not converge to a Nash equilibrium under merely the pseudo-monotonicity assumption. However, a closer look at the required conditions following Equation \eqref{eq:coco} reveals that the proof remains valid if instead of strictly monotone, the game mapping is a) pseudo-monotone and additionally satisfies b) $ \forall \bmu \in \Ab$ and $\bmu^*$ Nash equilibrium, $(\Mb(\bmu'),\bmu'-\bmu^*)=0 \Rightarrow \bmu' = \bmu^*$. This latter condition holds for example, when $\Mb$ is pseudo-monotone and co-coercive. Hence, our new proof method of Theorem \ref{th:main2} corrects  our mistake in \cite{Tat_ifac2017}. }

\mk{The challenge in generalizing the proof of Theorem \ref{th:main} to non-potential games lies in the fact that the extended game mapping $\Mb^0$ cannot be strictly monotone, nor can it be co-coercive. This implies that convergence of a subsequence of $\{ \be(t) \}$ will not suffice to establish convergence of the sequence to a Nash equilibrium. Nevertheless, given a strictly convex potential  function in the game we could use equivalence of the variational Nash equilibrium to an optimal primal dual pair for the Lagrangian, and establish convergence of the sequence of iterates to the variational Nash equilibrium. }
\end{rem}

\subsection{Convergence rate of the algorithm}

Below, we show that if the strict monotonicity condition for the game mapping in Assumption~\ref{assum:convex1} is strengthened to strong monotonicity, we obtain  a convergence rate for the procedure \eqref{eq:nocoupling} as a function of the stepsize and variance parameters. 
\begin{theorem}\label{th:convrate}
 Let Assumptions\r\ref{assum:CG_grad}-\ref{assum:inftybeh}, and \ref{assum:convex1} hold and $\Mb$ be strongly monotone with strong monotonicity constant $\kappa>1$. Furthermore, assume the time step and variance parameters are chosen as $\gamma_i(t)=\frac{1}{(t+R_i)^a}$, $\sigma_i(t)=\frac{1}{(t+R_i)^b}$, $i\in[N]$,  where $a+2b\in(0.5,1]$, $2a>1$, and $a+3b>1$. Then in the long run of the algorithm\r\eqref{eq:pbavmu}-\eqref{eq:pbavlam}
\begin{align}
\label{eq:conv_rate}
\E\{\|\bmu(t)-\bmu^*\|^2\}\le \frac{C}{t^{2(a+b)-1}} = O(1/{t^{2(a+b)-1}}),
\end{align}
 where $C$ is some positive constant and $\bmu^*$ is the unique equilibrium of the game $\Gamma$ to which the vector $\bmu(t)$ converges almost surely.
 \end{theorem}

 \begin{proof}
 We we will use the following lemma.
\begin{lem}\label{lem:app}
  Let the sequence $\{a_{t}\}$, $a_t\ge 0$ $t\in \Z_{+}$, satisfy the following iteration:
  $$a_{t+1}\le(1-\kappa/t)a_t+\psi/{t^{c}},$$
  for some constants $1<c\le 2$, $\kappa>1$, $\psi>0$.
  Then, $a_{t}\le \frac{C}{t^{c-1}},$
  where $C=\max\{a_0, \frac{\psi}{\kappa-1}\}$.
 \end{lem}
Please see the appendix for the proof. 

Using estimation \eqref{eq:f_noncoupl} and the fact that $(\Mb(\bmu^*), \bmu(t)-\bmu^*)\ge 0$ for any $t$, we obtain
\begin{align}\label{eq:mueast}
  \E\{&\|\bmu(t+1)-\bmu^*\|^2|\EuScript F_t\}\le\|\bmu(t)  - \bmu^*\|^2 \cr
&-2\bemin(t+1)(\Mb(\bmu(t))-\Mb(\bmu^*),\bmu(t)-\bmu^*) \cr
&\qquad\qquad\qquad\qquad+ O(s(t))\cr
&\le\left(1-\frac{\kappa}{t}\right)\|\bmu(t)- \bmu^*\|^2+\frac{\psi}{t^{2(a+b)}},
 \end{align}
 where in the last inequality we used the strong monotonicity of the map $\Mb$ over $\Ab$, conditions for the parameters $\gamma_i(t)$, $\sigma_i(t)$, and definition of $\beta_i$, which implies that there exists $t_0$ such  that $\bemin(t)\ge \frac{1}{2t}$ for any $t\ge t_0$.
Finally, taking into account that $1<2(a+b)\le 2$ and using Lemma\r\ref{lem:app} for $c=2(a+b)$, we conclude that  $\E\{\|\bmu(t)-\bmu^*\|^2\}\le {C}/{t^{2(a+b)-1}}$, where $C=\max\{\E\{\|\bmu(0)-\bmu^*\|^2\}, \frac{\psi}{\kappa-1}\}$.
\end{proof}
The result above demonstrates that the convergence rate of the proposed payoff-based learning procedure is sublinear. This is consistent with the results on related optimization algorithms  based on the stochastic approximation techniques \cite{juditsky2011, TatTouri}. Note also that Theorem\r\ref{th:convrate} presents the asymptotic estimation of the convergence rate, whereas its tightness and more details on characterization of the constant $C$ need to be analyzed separately and are subject of our future work.

\section{Numerical Case Study}\label{sec:simulations}

We illustrate the proposed payoff-based learning approach through a game arising in a classical Cournot economic model. There are $N$ firms, each producing a good and each needs to determine its production amount. Each firm (referred to also as a player or an agent) has an individual production cost $Q^i(\ab^i)$ and  receives a payment $p( \ab) \ab^i$  for the quantity produced $\ab^i$. The price $p(\ab)$ depends on the total production of all firms. The production of the firms is coupled by the fact that there is a network capacity constraint \cite{abolhassani2014network}. Such constraints may arise from the amount that can be delivered through a link to the consumers (consider for example an electricity network with line limits). In contrast to past approaches on computing Nash equilibria, we consider a  scenario in which the form of the price function $p$ is unknown to agents and there is no communication graph between the agents.

Let $\ab^i = [{a}^i_1,\ldots, {a}^i_d]^{\top}\in\mathbb R^d$
denote the decision variable of  firm $i$ (also referred to as player or agent), $i\in[N]$, which is its production level over a horizon of $d$ steps\footnote{The formulation here also can be interpreted as production levels of each firm at $d$ different locations \cite{abolhassani2014network}.}.  Each player has a limit on maximum production at each step
\begin{align}\label{eq:constraints}
 0\le &a^i_k\le \bar{a}^i \quad \mbox{for $k = 1,\ldots,d$}.
\end{align}
The convex and compact set defined by the constraints above is considered the action set $A_i$ for player $i$. The coupling constraints arising from a  network capacity limit is
\begin{align}\label{eq:coupledconstraints}
 \sum_{i=1}^N&a_k^i \le \bar{a}^k\quad \mbox{for $k = 1,\ldots,d$}.
\end{align}

For the simulations, we consider a linear price function and quadratic production cost functions, which are standard assumptions in Cournot models \cite{kukushkin1993cournot, hobbs1998lcp,ma2010decentralized,dario2016cdc}.
The  function to be minimized by each agent can then be compactly written as
\begin{align}\label{eq:costs}
 J_i(\ab^i,\ab^{-i}) &= Q^i(\ab^i) - p( \ab) \ab^i \\
 \nonumber
 &=\ab^{i\top}Q^i\ab^i + 2(C\frac{1}{N}\sum_{j=1}^N\ab^j+\cb)^{\top}\ab^i,
\end{align}
with $Q^i, C\in\mathbb R^{d\times d}$, $\cb\in\mathbb R^d$ for all $i\in[N]$.
The production cost is assumed convex and hence $Q^i \in\R^n_{+}$, whereas $C \in\R^n_{+}$ follows from the fact that the price is a decreasing function of total production \cite{hobbs1998lcp,weibull2006price,abolhassani2014network,kukushkin1993cournot,metzler2003nash,FaccPang1}. 
It is readily  verified that the resulting game mapping (see Definition \eqref{eq:gamemapping})  is affine and, hence, Lipschitz on $\R^{Nd}$. \mk{ Moreover, the game mapping is symmetric positive definite and hence, the game admits a strongly convex potential function}. Consequently, Assumptions\r\ref{assum:convex}-\ref{assum:potential} hold.

We let the agents follow the payoff-based algorithm described by \eqref{eq:regpl}-\eqref{eq:dualpl} to find their Nash equilibrium strategies. Each  player  submits  its  proposed  production
profile over time horizon of $d$ units, $\xb^i(t) = [{x}^i_1(t), \ldots, {x}^i_d(t)]^{\top}$ at iteration $t$. It then observes $J^0_i$ consisting of the cost functions corresponding to prices of the good,  the violation of the coupling constraint, as well as its  individual production cost.

\begin{figure}[t!]
\centering
\psfrag{0.5}[c][l]{\tiny{$0.5$}}
\psfrag{1.5}[c][l]{\tiny{$1.5$}}
\psfrag{2.5}[c][l]{\tiny{$2.5$}}
\psfrag{1}[c][c]{\tiny{$1$}}
\psfrag{2}[c][c]{\tiny{$2$}}
\psfrag{0}[c][c]{\tiny{$0$}}
\psfrag{3}[c][c]{\tiny{$3$}}
\psfrag{4}[c][c]{\tiny{$4$}}
\psfrag{5}[c][b]{\tiny{$5$}}
\psfrag{6}[c][b]{\tiny{$6$}}
\psfrag{7}[c][b]{\tiny{$7$}}
\psfrag{100}[c][b]{\tiny{$100$}}
\psfrag{200}[c][b]{\tiny{$200$}}
\psfrag{300}[c][b]{\tiny{$300$}}
\psfrag{50}[c][b]{\tiny{$50$}}
\psfrag{150}[c][b]{\tiny{$150$}}
\psfrag{250}[c][b]{\tiny{$250$}}
\psfrag{t}[c][b]{\footnotesize{}}
\psfrag{A}[c][b]{\footnotesize{}}
\begin{overpic}[width=1\linewidth]{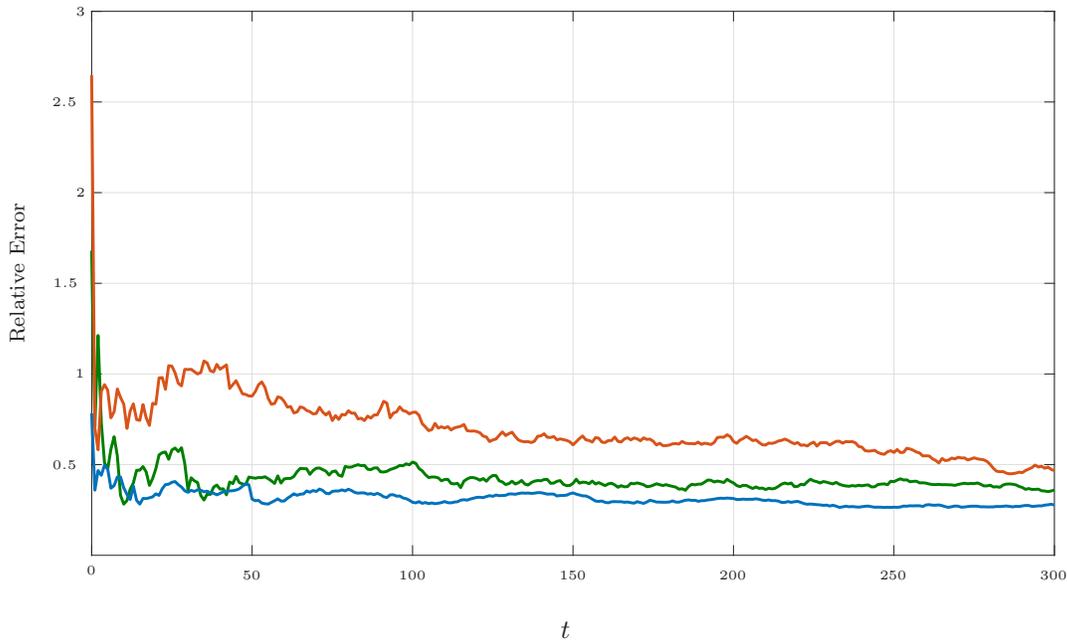}
\put(50.8,-0.8){$t$}
\put(6.4,23){\rotatebox{90}{\footnotesize{Relative Error}}}
\end{overpic}
\caption{Relative error $\frac{\|\bmu(t)-\ab^*\|}{\|\ab^*\|}$ during the payoff-based algorithm, $N=3$ (blue line), $N=10$ (green line), $N=30$ (red line).}
\label{fig:elmark}
\end{figure}

For the simulation, we let $d=4$, the matrices $Q^i$, $i\in[N]$ and $C$,  in \eqref{eq:costs} are the identity matrices in $\R^{d
\times d}$, and the vector $\cb \in
\R^d$ is chosen randomly according to a normal distribution. The action set $A_i$ for each player $i\in[N]$ is defined by \eqref{eq:constraints}-\eqref{eq:coupledconstraints}, where $\bar a^i = 9$ and $\bar{a}^k$ is a random variable taking values in the interval $(3N,3N+100)$. {The initial vector $(\bmu(0),\bla(0))$ is chosen from a uniform distribution on $\Ab \times [0,5]$.

Figure \ref{fig:elmark} presents the relative error $\frac{\|\bmu(t)-\ab^*\|}{\|\ab^*\|}$ during the algorithm's run for $N=3, 10, 30$, where $\ab^*$ is the unique generalized Nash equilibrium of the game. We  see that after the first iteration the iterates quickly approach the Nash equilibrium.  However, convergence of the error to zero is slow. The slow decrease of the relative error after the first iteration can be explained by the choice of the rapidly decreasing parameter $\sigma(t)$ and $\gamma(t)$. The convergence is also slower for increasing  number of players in the game. It is interesting to derive explicit dependence of the convergence rate derived based on the number of players. 

\section{Conclusion}\label{sec:conclusion}
We proposed a novel payoff-based learning approach for convergence to variational Nash equilibria in convex games with jointly convex coupling constraints. In this approach, each agent determined its next state by sampling from a Gaussian distribution, whose mean was updated using the payoff information. We proved almost sure convergence of the means of the distributions to a variational Nash equilibrium, given appropriate choice of algorithms' step-sizes and  variances of the distributions. The convergence result relied on existence of a strictly convex potential function. In the absence of coupling constraint, convergence to a Nash equilibrium was established based on strict monotonicity of the game mapping. Furthermore, in this case,  under strong monotonicity of the game mapping, the convergence rate of the algorithm was derived. 
Further relaxing  conditions for convergence of the payoff-based algorithm with and without coupling constraint is subject of our current work. We are also  developing algorithms that ensure constraint satisfaction during the algorithm iterates. 

\section{Acknowledgement} 
We thank Sergio Grammatico for informing us of  the recent publication \cite{grammatico2018comments}, in which it was shown that monotonicity is not a sufficient condition for convergence of a class of forward-backward algorithms. This implied a bug in our convergence proof for the payoff-based algorithm proposed in \cite{Tat_ifac2017}. In the current paper,  we corrected this mistake by deriving stronger convergence conditions in Theorem \ref{th:main2}.

\bibliographystyle{plain}
\bibliography{TAC_ArxivVersion}

\appendix
\subsection{Proofs}
\begin{proof}[Proof of Lemma 1]
Assume the set $Y$ is compact and is expressed by the following inequality constraints:
$Y=\{\yb\in\R^d: \hb(\yb)=[h_1(\yb),\ldots,h_m(\yb)]^T\le \boldsymbol 0 \}$, where $h_i:\R^d\to\R$, $i\in[m]$, are convex functions defined on $R^d$.
Then, using Slater's condition for the set $Y$ and continuity of $\boldsymbol T$, conditions analogous to the Karush-Kuhn-Tucker ones can be formulated for $SOL(Y,\boldsymbol T)$ (see Proposition 1.3.4 in \cite{FaccPang1}). Namely, $\yb^*\in SOL(Y,\boldsymbol T)$ if and only if there exists $\boldsymbol{\nu}\in\R^m$ such that
\begin{align}\label{eq:KKT}
 \boldsymbol 0 & = \boldsymbol T(\by^*) + \sum_{i=1}^m(\nu_i, \nabla h_i(\by^*)), \cr
 0&=(\boldsymbol{\nu}, \hb(\yb^*)), \quad \boldsymbol{\nu}\ge\boldsymbol 0, \quad \hb(\yb^*)\le\boldsymbol 0.
\end{align}
Thus, we can associate a multiplier $\boldsymbol {\nu}$ with any solution $\yb^*$ of $VI(Y,\boldsymbol T)$.  We further call the vector $(\yb^*,\boldsymbol{\nu})$ a \emph{KKT tuple}. 

Now consider the associated game $\Gamma_{a}(\Ab\times \R_+^{n})$ defined in \eqref{eq:assocgame1}. 
Note that the variational equilibria in game $\Gamma$ are characterized as $SOL(\EQ,\boldsymbol M)$. Hence, $\ab^*$ is a variational equilibrium in $\Gamma$ if and only if $(\ab^*, \boldsymbol{\nu})$ is a KKT tuple for $VI(\EQ,\boldsymbol M)$. And from Lemma\r3 in \cite{dario2016cdc}, under Assumptions\r\ref{assum:convex} and \ref{assum:Slaters}, $(\ab^*,\boldsymbol{\nu})$ is a KKT tuple for the $VI(\EQ, \Mb)$ if and only if the pair $[\ab^*, \bla^*]$ is a Nash equilibrium of the game $\Gamma_{a}(\Ab\times\R^n_{+})$, where $\bla^*$ denotes the coordinate of the multiplier $\boldsymbol {\nu}$ corresponding to the constraint $\gb(\ab)\le \boldsymbol 0$. Thus, Assertion 1) and 2) are proven. Furthermore, taking into account Propositions 1.3.4 and 1.4.2 in \cite{FaccPang1} such a KKT tuple exists since  $VI(\EQ, \Mb)$ has a solution. It follows that there exists a Nash equilibrium in $\Gamma_{a}(\Ab\times\R^n_{+})$ and Assertion 3) is proven. Finally, Assertion 4) holds since  under Assumptions\r\ref{assum:convex} and \ref{assum:Slaters},  Lemma\r5.1 in \cite{ZhuFrazzoli} shows that $\|  \bla^* \|$ is bounded.

\end{proof}

%

\begin{proof}[Proof of Lemma \ref{lem:app}]
The proof is  based on a standard rate estimate, analogous to the result in (5.292) in \cite{stochprogr}.
 We prove the claim by induction. Let us assume that $a_0\ge \frac{\psi}{\kappa-1}$. Then, according to the induction step,
 $$a_{t+1}\le \left(1-\frac{\kappa}{t}\right)\frac{a_0}{t^{c-1}}+\frac{\psi}{t^{c}}.$$
 Thus, it suffices to show that the right-hand-side of the inequality above is not more than $\frac{a_0}{(t+1)^{c-1}}$. As $\frac{\psi}{a_0}-\kappa\le-1$,
 \begin{align*}
  \left(1-\frac{\kappa}{t}\right)\frac{a_0}{t^{c-1}}+\frac{\psi}{t^{c}}=\frac{a_0}{t^{c}}\left(t-\kappa+\frac{\psi}{a_0}\right)&\le\frac{a_0}{t^{c}}(t-1)\cr
  &\le \frac{a_0}{(t+1)^{c-1}},
 \end{align*}
 since $a_0>0$ and $\left(1+\frac{1}{t}\right)^{c-1}\le\left(1+\frac{1}{t-1}\right).$
 The case $a_0\le \frac{\psi}{\kappa-1}$ can be considered analogously.
\end{proof}

\subsection{Supporting Theorems}
Let $\{\zbx(t)\}_t$, $t\in \Z_+$, be a discrete-time Markov process on some state space $E\subseteq \R^d$, namely $\zbx(t)=\zbx(t,\omega):\Z_+\times\Omega\to E$, where $\Omega$ is the sample space of the probability space on which the process $\zbx(t)$ is defined. The transition function of this chain, namely $\Pr\{\zbx(t+1)\in\Gamma| \zbx(t)=\zbx\}$, is denoted by $P(t,\zbx,t+1,\Gamma)$, $\Gamma\subseteq E$.

\begin{definition}\label{def:def1}
The operator $L$ defined on the set of measurable functions $V:\Z_+\times E\to \R$, $\zbx\in E$, by
\begin{align*}
LV(t,\zbx)&=\int{P(t,\zbx,t+1,dy)[V(t+1,y)-V(t,\zbx)]}\cr
&=E[V(t+1,\zbx(t+1))\mid \zbx(t)=\zbx]-V(t,\zbx),
\end{align*}
is called a \emph{generating operator} of a Markov process $\{\zbx(t)\}_t$.
\end{definition}
Next, we formulate the following theorem for discrete-time Markov processes, which is proven in \cite{NH}, Theorem\r2.5.2.

\begin{theorem}\label{th:finiteness}
  Consider a Markov process $\{\zbx(t)\}_t$ and suppose that there exists a function $V(t,\zbx)\ge 0$ such that $\inf_{t\ge0}V(t,\zbx)\to\infty$ as $\|\zbx\|\to\infty$ and
  \[LV(t,\zbx)\le -\alpha(t+1)\psi(t,\zbx) + f(t)(1+V(t,\zbx)),\]
   where $\psi\ge 0$ on $\R \times \R^d$, $f(t)>0$, $\sum_{t=0}^{\infty}f(t)<\infty$. Let $\alpha(t)$ be such that $\alpha(t)>0$, $\sum_{t=0}^{\infty} \alpha(t)= \infty$.
   Then, almost surely $\sup_{t\ge 0}\|\zbx(t,\omega)\| = R(\omega)< \infty$.
\end{theorem}

The following is a well-known result of Robbins and Siegmund on non-negative random variables \cite{robbins1985convergence}.
\begin{theorem}\label{th:th_nonnegrv} Let $(\Omega, F, P)$ be a probability space and $F_1\subset F_2\subset\dots$ a sequence of sub-$\sigma$-algebras of $F$.
 Let $z_t, b_t, \xi_t,$ and $\zeta_t$ be non-negative $F_t$-measurable random variables satisfying
 \begin{align*}
  \E(z_{t+1}|F_t)\le z_t(1+b_t)+\xi_t-\zeta_t.
 \end{align*}
Then, almost surely $\lim_{t\to\infty} z_t$ exists and is finite for the case in which $\{\sum_{t=1}^{\infty}b_t<\infty, \;\sum_{t=1}^{\infty}\xi_t<\infty\}$. Moreover, in this case, $\sum_{t=1}^{\infty}\zeta_t<\infty$ almost surely. 
\end{theorem}

%

\end{document}